\documentclass{article}

\usepackage[latin1,applemac]{inputenc}
\usepackage[T1]{fontenc}
\usepackage[french,english]{babel}
\usepackage{amssymb,amsmath,amsthm,color,bbm}
\usepackage{stmaryrd}
\usepackage{dsfont}
\usepackage{graphicx}
\usepackage{color}
\usepackage{subfig}
\usepackage{version}
\usepackage{ifthen}
\usepackage{enumerate}
\usepackage{algorithm,algorithmic}
\usepackage{slashbox}

\theoremstyle{plain} \newtheorem{proposition}{Proposition}
\theoremstyle{plain} \newtheorem{theorem}{Theorem}
\theoremstyle{plain} 
\theoremstyle{plain} \newtheorem{lemma}{Lemma}
\theoremstyle{remark} \newtheorem{remark}{Remark}

\newcommand{\normL}[2]{\left\| #2 \right\|_{#1}}
\newcommand{\normLpw}[2]{\mathbb{E}\left[\left| #2 \right|^{#1}\right]}

\newcommand{\addfunc}[1]{S_{#1}}

\newcommand{\osc}{\mathrm{osc}}
\newcommand{\size}{r}

\newcommand{\eqdef}{\ensuremath{\stackrel{\mathrm{def}}{=}}}

\newcommand{\GNorm}{\cite[Lemma 10]{douc:garivier:moulines:olsson:2010}}

\newcommand{\TechLem}{\cite[Lemma 4]{douc:garivier:moulines:olsson:2010}}
\newcommand{\GisCentered}{\cite[Lemma 3]{douc:garivier:moulines:olsson:2010}}
\def\Xset{\mathbb{X}}
\def\Yset{\mathbb{Y}}

\def\sigmaX{\mathcal{X}}
\def\sigmaY{\mathcal{Y}}

\newcommand{\error}{\Delta_T^N[\addfunc{T,\size}]}
\newcommand{\errorh}{\Delta_T^N[h]}

\newcommand{\errorsmall}{\delta_{T}^N}
\newcommand{\eqsp}{\;}
\newcommand{\epart}[2]{\xi_{#1}^{#2}}
\newcommand{\ewght}[2]{\omega_{#1}^{#2}}
\newcommand{\rmd}{\mathrm{d}}
\newcommand{\rmL}{\mathrm{L}}
\newcommand{\rmE}{\mathrm{E}}
\newcommand{\rme}{\mathrm{e}}
\newcommand{\rmB}{\mathrm{B}}

\newcommand{\maxosc}[1]{\underset{\size\leq t\leq #1}{\max}\left\{\osc(h_t)\right\}}

\newcommand{\Xinit}{\ensuremath{\chi}}
\newcommand{\XinitIS}[2][]{\ifthenelse{\equal{#1}{}}{\ensuremath{\rho_{#2}}}{\ensuremath{\check{\rho}_{#2}}}}
\newcommand{\iid}{i.i.d.}
\newcommand{\CExp}[2]{\mathbb{E}\left[#1\middle | #2\right]}

\newcommand{\filt}[2][]%
{
\ifthenelse{\equal{#1}{}}{\ensuremath{\phi_{#2}}}%
{\ifthenelse{\equal{#1}{hat}}{\ensuremath{\phi^{N}_{#2}}}
{\ifthenelse{\equal{#1}{tilde}}{\ensuremath{\tilde{\phi}^{N}_{#2}}}
{\ifthenelse{\equal{#1}{tar}}{\ensuremath{\phi^{N,\mathrm{t}}_{#2}}}
{\ifthenelse{\equal{#1}{aux}}{\ensuremath{\phi^{N,\mathrm{a}}_{#2}}}
}
}
}
}
}
\newcommand{\mcbf}[2][]{%
\ifthenelse{\equal{#1}{}}{\overline{\mathcal{F}}_{#2}}{\overline{\mathcal{F}}_{#2}^{(#1)}}%
}

\newcommand{\unfilt}[2][]%
{
\ifthenelse{\equal{#1}{}}{\ensuremath{\gamma_{#2}}}%
{\ifthenelse{\equal{#1}{hat}}{\ensuremath{\gamma^{N}_{#2}}}
{\ifthenelse{\equal{#1}{tilde}}{\ensuremath{\tilde{\gamma}^{N}_{#2}}}
{\ifthenelse{\equal{#1}{tar}}{\ensuremath{\gamma^{N,\mathrm{t}}_{#2}}}
{\ifthenelse{\equal{#1}{aux}}{\ensuremath{\gamma^{N,\mathrm{a}}_{#2}}}
}
}
}
}
}
\newcommand{\CPE}[3][]
{\ifthenelse{\equal{#1}{}}{\mathbb{E}\left[\left. #2 \, \right| #3 \right]}{\mathbb{E}_{#1}\left[\left. #2 \, \right| #3 \right]}}
\newcommand{\mcf}[2]{\mathcal{F}_{#1}^{#2}}
\newcommand{\1}{\ensuremath{\mathbf{1}}}

\newcommand{\Xsigma}[1][]%
{%
\ifthenelse{\equal{#1}{}}{\ensuremath{\mathcal{B}(\Xset)}}{\ensuremath{\mathcal{B}(\Xset^{#1})}}
}
\newcommand{\sumwght}[2][]{%
\ifthenelse{\equal{#1}{}}{\ensuremath{\Omega_{#2}}}{\ensuremath{\Omega_{#2}^{#1}}}}

\newcommand{\instrpostaux}[2]{\ensuremath{\pi_{#1|#2}}}
\newcommand{\adjfunc}[4][]
{\ifthenelse{\equal{#1}{}}{\ifthenelse{\equal{#4}{}}{\vartheta_{#2}}{\vartheta_{#2}(#4)}}
{\ifthenelse{\equal{#1}{smooth}}{\ifthenelse{\equal{#4}{}}{\tilde{\vartheta}_{#2}}{\tilde{\vartheta}_{#2}(#4)}}
{\ifthenelse{\equal{#1}{fully}}{\ifthenelse{\equal{#4}{}}{\vartheta^\star_{#2}}{\vartheta^\star_{#2}(#4)}}{\mathrm{erreur}}}}}
\newcommand{\chunk}[4][]%
{\ifthenelse{\equal{#1}{}}{\ensuremath{{#2}_{#3:#4}}}{\ensuremath{#2^#1}_{#3:#4}}
}
\newcommand{\kiss}[3][]
{\ifthenelse{\equal{#1}{}}{p_{#2}}
{\ifthenelse{\equal{#1}{fully}}{p^{\star}_{#2}}
{\ifthenelse{\equal{#1}{smooth}}{\tilde{r}_{#2}}{\mathrm{erreur}}}}}

\newcommand{\Kiss}[3][]
{\ifthenelse{\equal{#1}{}}{P_{#2}}
{\ifthenelse{\equal{#1}{fully}}{P^{\star}_{#2}}
{\ifthenelse{\equal{#1}{smooth}}{\tilde{R}_{#2}}{\mathrm{erreur}}}}}
\newcommand{\m}{\ensuremath{m}}
\newcommand{\BK}[1]{\mathrm{B}_{#1}}
\newcommand{\post}[3][]%
{
\ifthenelse{\equal{#1}{}}{\ensuremath{\phi_{#2|#3}}}%
{\ifthenelse{\equal{#1}{hat}}{\ensuremath{\phi^{N}_{#2|#3}}}
{\ifthenelse{\equal{#1}{tilde}}{\ensuremath{\tilde{\phi}^{N}_{#2|#3}}}
{\ifthenelse{\equal{#1}{tar}}{\ensuremath{\phi^{N,\mathrm{t}}_{#2|#3}}}}
}
}
}

\newcounter{hypA}
\newenvironment{hypA}{\refstepcounter{hypA}\begin{itemize}
\item[{\bf A\arabic{hypA}}]}{\end{itemize}}

\begin{document}

\title{Non-asymptotic deviation inequalities for smoothed additive functionals in non-linear state-space models.}

\author{Cyrille Dubarry\footnote{SAMOVAR, Institut T\'el\'ecom/T\'el\'ecom SudParis, France, cyrille.dubarry@it-sudparis.eu} \,and Sylvain Le Corff\footnote{LTCI, Institut T\'el\'ecom/T\'el\'ecom ParisTech, France, sylvain.lecorff@telecom-paristech.fr}\footnote{This work was supported by the French National Research Agency, under the program ANR-07 ROBO 0002}}

\maketitle

\begin{abstract}
The approximation of fixed-interval smoothing distributions is a key issue in inference for general state-space hidden Markov models (HMM). This contribution establishes non-asymptotic bounds for  the Forward Filtering Backward Smoothing (FFBS)  and the Forward Filtering Backward Simulation (FFBSi) estimators of fixed-interval smoothing functionals. We show that the rate of convergence of the $\rmL_{q}$-mean errors of both methods depends on the number of observations $T$ and  the number of particles $N$ only through the ratio $T/N$ for additive functionals. In the case of the FFBS, this improves recent results providing bounds depending on $T/\sqrt{N}$.
 \end{abstract}

\section{Introduction}
\label{sec:intro}
State-space models play a key role in statistics, engineering and econometrics; see \cite{cappe:moulines:ryden:2005,durbin:koopman:2000,west:harrison:1989}. Consider a process $\{X_{t}\}_{t\geq 0}$ taking values in a general state-space $\Xset$. This hidden process can be observed only through the observation process $\{Y_{t}\}_{t\geq 0}$ taking values in $\Yset$. Statistical inference in general state-space models involves the computation of expectations of additive functionals of the form
\[
S_{T} = \sum_{t=1}^{T}h_{t}(X_{t-1},X_{t})\eqsp,
\]
conditionally to $\{Y_{t}\}_{t=0}^{T}$, where $T$ is a positive integer and $\{h_{t}\}_{t=1}^{T}$ are functions defined on $\Xset^{2}$. These smoothed additive functionals appear naturally for maximum likelihood parameter inference in hidden Markov models. The computation of the gradient of the log-likelihood function (Fisher score) or of the intermediate quantity of the Expectation Maximization algorithm involves the estimation of such smoothed functionals, see \cite[Chapter $10$ and $11$]{cappe:moulines:ryden:2005} and \cite{doucet:poyiadjis:singh:2009}.

Except for linear Gaussian state-spaces or for finite state-spaces, these smoothed additive functionals cannot be computed explicitly.
In this paper, we consider Sequential Monte Carlo algorithms, henceforth referred to as particle methods, to approximate these quantities. These methods combine sequential importance sampling and sampling importance resampling steps to produce a set of random particles with associated importance weights to approximate the fixed-interval smoothing distributions.

The most straightforward implementation is based on the so-called path-space method. The complexity of this algorithm per time-step grows only linearly with the number $N$ of particles, see \cite{delmoral:2004}. However, a well-known shortcoming of this algorithm is known in the literature as the path degeneracy; see \cite{doucet:poyiadjis:singh:2009} for a discussion.

Several solutions have been proposed to solve this degeneracy problem.  In this paper, we consider the Forward Filtering Backward Smoothing algorithm (FFBS) and  the Forward Filtering Backward Simulation algorithm (FFBSi) introduced in \cite{doucet:godsill:andrieu:2000} and further developed in \cite{godsill:doucet:west:2004}. Both algorithms proceed in two passes. In the forward pass, a set of particles and weights is stored. In the Backward pass of  the FFBS the weights are modified but the particles are kept fixed. The FFBSi  draws independently different particle trajectories among all possible paths.
Since they use a backward step, these algorithms are mainly adapted for batch estimation problems. However, as shown in \cite{delmoral:doucet:singh:2010}, when applied to additive functionals, the FFBS algorithm can be implemented forward in time, but its complexity grows quadratically with the number of particles. As shown in \cite{douc:garivier:moulines:olsson:2010}, it is possible to implement the FFBSi with a complexity growing only linearly with the number of particles.

The control of the $\rmL_{q}$-norm of the deviation between the smoothed additive functional and its particle approximation has been studied recently in \cite{delmoral:doucet:singh:2010,delmoral:doucet:singh:2010b}. In an unpublished paper by \cite{delmoral:doucet:singh:2010b}, it is shown that the FFBS estimator variance of any smoothed additive functional is upper bounded by terms depending on $T$ and $N$ only through the ratio $T/N$. Furthermore, in \cite{delmoral:doucet:singh:2010}, for any $q > 2$, a $\rmL_q$-mean error bound for smoothed functionals computed with the FFBS is established. When applied to strongly mixing kernels, this bound amounts to be of order $T/\sqrt{N}$ either for
\begin{enumerate}[(i)]
\item \label{item:general} uniformly bounded in time general path-dependent functionals,
\item \label{item:add} unnormalized additive functionals (see \cite[Eq.~(3.8), pp.~957]{delmoral:doucet:singh:2010}).
\end{enumerate}

In this paper, we establish $\rmL_q$-mean error and exponential deviation inequalities of both the FFBS and FFBSi smoothed functionals estimators.
We show that, for any $q \geq 2$, the $\rmL_q$-mean error for both algorithms is upper bounded by terms depending on $T$ and $N$ only through the ratio $T/N$ under the strong mixing conditions for \eqref{item:general} and \eqref{item:add}. We also establish an exponential deviation inequality with the same functional dependence in $T$ and $N$.

This paper is organized as follows. Section~\ref{sec:framework} introduces further definitions and notations and the FFBS and FFBSi algorithms. In Section~\ref{sec:results}, upper bounds for the $\rmL_q$-mean error and exponential deviation inequalities of these two algorithms are presented. In Section \ref{sec:MCexp}, some Monte Carlo experiments are presented to support our theoretical claims. The proofs are presented in Sections~\ref{sec:proof:theorem:mainresult} and \ref{sec:proof:theorem:expineq}.

\section{Framework}
\label{sec:framework}
Let $\Xset$ and $\Yset$ be two general state-spaces endowed with countably generated $\sigma$-fields $\sigmaX$ and $\sigmaY$. Let $M$ be a Markov transition kernel defined on $\Xset\times\sigmaX$ and $\{g_{t}\}_{t\geq 0}$ a family of functions defined on $\Xset$. It is assumed that, for any $x \in \Xset$, $M(x,\cdot)$ has a density $m(x, \cdot)$ with respect to a reference measure $\lambda$ on $(\Xset,\sigmaX)$. For any integers $T\geq 0$ and $0 \leq s \leq t \leq T$, any measurable function $h$ on $\Xset^{t-s+1}$, and any probability distribution $\chi$ on $(\Xset,\sigmaX)$, define
\begin{equation}
\label{eq:smooth}
\phi_{s:t|T}[h] \eqdef \frac{\int \chi(\rmd x_0)g_{0}(x_0)\prod_{u=1}^{T}M(x_{u-1},\rmd x_u)g_{u}(x_u)h(x_{s:t})}{\int \chi(\rmd x_0)g_{0}(x_0)\prod_{u=1}^{T}M(x_{u-1},\rmd x_u)g_{u}(x_u)}\eqsp,
\end{equation}
where  $a_{u:v}$ is a short-hand notation for $\{a_s\}_{s=u}^{v}$. The dependence on $g_{0:T}$ is implicit and is dropped from the notations.
\begin{remark}
Note that this equation has a simple interpretation in the particular case of hidden Markov models. Indeed, let $\left(\Omega, \mathcal{F}, \mathbb{P}\right)$ be a probability space and $\{X_{t}\}_{t\ge 0}$ a Markov chain on $\left(\Omega, \mathcal{F}, \mathbb{P}\right)$ with transition kernel $M$ and initial distribution $\chi$ (which we denote $X_{0}\sim\chi$). Let $\{Y_{t}\}_{t\ge 0}$ be a sequence of observations  on $\left(\Omega, \mathcal{F}, \mathbb{P}\right)$ conditionally independent given $\sigma(X_{t},t\ge 0)$ and such that the conditional distribution of $Y_{u}$ given $\sigma(X_{t},t\ge 0)$ has a density given by $g(X_{u},\cdot)$ with respect to a reference measure on $\sigmaY$ and set $g_{u}(x) = g(x,Y_{u})$. Then, the quantity $\phi_{s:t|T}[h] $ defined by \eqref{eq:smooth} is the conditional expectation of $h(X_{s:t})$ given $Y_{0:T}$:
\[
\phi_{s:t|T}[h]  = \mathbb{E}\left[h(X_{s:t})\middle|Y_{0:T}\right]\eqsp,\quad X_{0}\sim\chi\eqsp.
\]
\end{remark}
In its original version, the FFBS algorithm proceeds in two passes.  In the forward pass, each filtering distribution $\filt{t}\eqdef\filt{t:t}$, for any $t\in\{0,\dots,T\}$,  is approximated using weighted samples $\left\{(\ewght{t}{N,\ell},\epart{t}{N,\ell})\right\}_{\ell=1}^{N}$, where $T$ is the number of observations and $N$ the number of particles: all sampled particles and weights are stored. In the backward pass of the FFBS, these importance weights are then modified (see \cite{doucet:godsill:andrieu:2000,huerzeler:kuensch:1998,kitagawa:1996}) while the particle positions are kept fixed. The importance weights are updated recursively backward in time to obtain an approximation of the fixed-interval smoothing distributions $\left\{\post{s:T}{T}\right\}_{s=0}^{T}$.  The particle approximation is constructed as follows.
\paragraph{Forward pass}
Let $\{\epart{0}{N,\ell}\}_{\ell = 1}^N$ be \iid\ random variables distributed according to the instrumental density $\rho_0$ and set the importance weights $\ewght{0}{N,\ell} \eqdef \rmd \Xinit / \rmd \XinitIS{0}(\epart{0}{N,\ell}) \, g_{0}(\epart{0}{N,\ell})$. The weighted sample $\{ (\epart{0}{N,\ell}, \ewght{0}{N,\ell})\}_{\ell = 1}^N$ then targets the initial filter $\filt{0}$ in the sense that
$\filt[hat]{0}[h]\eqdef \sum_{\ell=1}^N \ewght{0}{N,\ell}h(\epart{0}{N,\ell})/\sum_{\ell=1}^N \ewght{0}{N,\ell}$ is a consistent estimator of $\filt{0}[h]$ for any bounded and measurable function  $h$ on $\Xset$.\\
Let now $\{ (\epart{s-1}{N,\ell}, \ewght{s-1}{N,\ell}) \}_{\ell = 1}^N$ be a weighted sample targeting $\filt{s-1}$. We aim at computing new particles and importance weights targeting the probability distribution
$\filt{s}$. Following \cite{pitt:shephard:1999}, this may be done by simulating pairs $\{ (I_s^{N,\ell}, \epart{s}{N,\ell}) \}_{\ell = 1}^N$ of indices and particles from the instrumental distribution:
\begin{equation*} \label{eq:instrumental-distribution-filtering}
    \instrpostaux{s}{s}(\ell, h) \propto \ewght{s-1}{N,\ell} \adjfunc{s}{s}{\epart{s-1}{N,\ell}} \Kiss{s}{s}(\epart{s-1}{N,\ell},h) \eqsp,
\end{equation*}
on the product space $\{1, \dots, N\} \times \Xset$, where $\{ \adjfunc{s}{s}{\epart{s-1}{N,\ell}} \}_{\ell = 1}^N$ are the adjustment multiplier weights and $\Kiss{s}{s}$
is a Markovian proposal transition kernel. In the sequel, we assume that $\Kiss{s}{s}(x, \cdot)$ has, for any $x \in \Xset$, a density $\kiss{s}{s}(x, \cdot)$
with respect to the reference measure $\lambda$. For any  $\ell \in\{1, \dots, N\}$ we associate to the particle $\epart{s}{N,\ell}$ its importance weight defined by:
\begin{equation*} \label{eq:weight-update-filtering}
    \ewght{s}{N,\ell} \eqdef \frac{\m(\epart{s-1}{N,I_s^{N,\ell}},\epart{s}{N,\ell}) g_{s}(\epart{s}{N,\ell})}{\adjfunc{s}{s}{\epart{s-1}{N,I_s^{N,\ell}}} \kiss{s}{s}(\epart{s-1}{N,I_s^{N,\ell}},\epart{s}{N,\ell})} \eqsp.
\end{equation*}
\paragraph{Backward smoothing}
For any probability measure $\eta$ on $(\Xset, \sigmaX)$, denote by $\BK{\eta}$ the backward smoothing kernel given, for all bounded measurable function $h$ on $\Xset$ and for all  $x\in\Xset$, by:
\begin{equation*} \label{eq:backward-kernel}
    \BK{\eta}(x, h) \eqdef \frac{\int \eta(\rmd x') \; \m(x', x) h(x')}
    {\int \eta(\rmd x') \; \m(x', x)} \eqsp,
\end{equation*}
For all $s\in\{0,\dots,T-1\}$ and for all bounded measurable function $h$ on $\Xset^{T-s+1}$, $\post{s:T}{T}[h]$ may be computed recursively, backward in time, according to
\begin{equation*} \label{eq:smoothing:backw_decomposition_recursion}
    \post{s:T}{T}[h] = \int \BK{\filt{s}}(x_{s+1}, \rmd x_s) \, \post{s+1:T}{T}(\rmd \chunk{x}{s+1}{T}) \, h(\chunk{x}{s}{T}) \eqsp.
\end{equation*}

\subsection{The forward filtering backward smoothing algorithm}
\label{subsec:FFBS}
Consider the weighted samples $\left\{ (\epart{t}{N,\ell}, \ewght{t}{N,\ell})\right\}_{\ell = 1}^N$, drawn for any $t\in\{0,\dots,T\}$ in the forward pass. An approximation of the fixed-interval smoothing distribution can be obtained using
\begin{equation} \label{eq:smoothing:backw_decomposition_recursion_sample}
    \post[hat]{s:T}{T}[h] = \int \BK{\filt[hat]{s}}(x_{s+1}, \rmd x_s) \, \post[hat]{s+1:T}{T}(\rmd \chunk{x}{s+1}{T}) \, h(\chunk{x}{s}{T}) \eqsp,\\
\end{equation}
and starting with $ \post[hat]{T:T}{T}[h] = \filt[hat]{T}[h]$.
Now, by definition, for all $x\in\Xset$ and for all bounded measurable function $h$ on $\Xset$,
$$
    \BK{\filt[hat]{s}}(x, h) = \sum_{i = 1}^N \frac{\ewght{s}{N,i} \m(\epart{s}{N,i}, x)}{\sum_{\ell = 1}^N \ewght{s}{N,\ell} \m(\epart{s}{N,\ell},x)}  h\left( \epart{s}{N,i} \right) \eqsp,
$$
and inserting this expression into \eqref{eq:smoothing:backw_decomposition_recursion_sample} gives the following particle approximation of the fixed-interval smoothing distribution $\post{0:T}{T}[h]$
\begin{equation} \label{eq:forward-filtering-backward-smoothing}
    \post[hat]{0:T}{T}\left[h\right] = \sum_{i_0 = 1}^N \dots \sum_{i_T = 1}^N
    \left(\prod_{u=1}^T \Lambda_u^N(i_u,i_{u-1})\right)  \times \frac{\ewght{T}{N,i_T}}{\sumwght[N]{T}} h\left(\epart{0}{N,i_0}, \dots, \epart{T}{N,i_T}\right) \eqsp,
\end{equation}
where $h$ is a bounded measurable function on $\Xset^{T+1}$,
\begin{equation} \label{eq:definition-transition-matrix-W}
    \Lambda^N_{t}(i,j) \eqdef \frac{\ewght{t}{N,j} \m(\epart{t}{N,j}, \epart{t+1}{N,i})}{\sum_{\ell=1}^N \ewght{t}{N,\ell} \m(\epart{t}{N,\ell}, \epart{t+1}{N,i})} \eqsp, \quad (i,j) \in \{1, \dots, N\}^2 \eqsp,
\end{equation}
 and
\begin{equation} \label{eq:defOmega}
    \sumwght[N]{t} \eqdef \sum_{\ell=1}^N \ewght{t}{N,\ell} \eqsp.
\end{equation}
The estimator of the fixed-interval smoothing distribution
$\post[hat]{0:T}{T}$ might seem impractical since the cardinality of its support is $N^{T+1}$. Nevertheless, for additive functionals of the form
\begin{equation}\label{eq:addfuncbis}
\addfunc{T,\size}(\chunk{x}{0}{T}) = \sum_{t=\size}^Th_t(x_{t-\size:t})\eqsp,
\end{equation}
where $r$ is a non negative integer and $\{h_t\}_{t=\size}^{T}$ is a family of bounded measurable functions on $\Xset^{\size+1}$, the complexity of the FFBS algorithm is reduced to $O(N^{\size+2})$. Furthermore, the smoothing of such functions can be computed forward in time as shown in \cite{delmoral:doucet:singh:2010}. This forward algorithm is exactly the one presented in \cite{doucet:poyiadjis:singh:2009} as an alternative to the use of the path-space method. Therefore, the results outlined in Section~\ref{sec:results} hold for this method and confirm the conjecture mentioned in \cite{doucet:poyiadjis:singh:2009}.

\subsection{The forward filtering backward simulation algorithm}
\label{subsec:FFBSi}
We now consider an algorithm whose complexity grows only linearly with the number of particles for any functional on $\Xset^{T+1}$. For any $t\in\{1,\dots,T\}$, we define
\begin{equation*} \label{eq:definition-mcf}
    \mcf{t}{N} \eqdef \sigma\left\{(\epart{s}{N,i}, \ewght{s}{N,i}); 0 \leq s \leq t, 1 \leq i \leq N\right\}\eqsp.
\end{equation*}
The transition probabilities $\{\Lambda_t^N\}_{t=0}^{T-1}$ defined in \eqref{eq:definition-transition-matrix-W} induce an inhomogeneous Markov chain $\{ J_{u} \}_{u = 0}^T$ evolving backward in time as follows. At time $T$, the random index $J_T$ is drawn from the set $\{1, \dots, N\}$ with probability proportional to $(\ewght{T}{N,1},\dots,\ewght{T}{N,N})$.
For any $t \in \{0, \dots, T-1\}$, the index $J_t$ is sampled in the set $\{1,\dots,N\}$ according to $\Lambda^N_{t}(J_{t+1},\cdot)$.
The joint distribution of $\chunk{J}{0}{T}$ is therefore given, for $\chunk{j}{0}{T} \in \{1, \dots, N\}^{T + 1}$, by
\begin{equation} \label{eq:distribution-non-homogeneous}
    \mathbb{P} \left[ \chunk{J}{0}{T} = \chunk{j}{0}{T} \left| \mcf{T}{N} \right. \right] = \frac{\ewght{T}{N,j_T}}{\sumwght[N]{T}} \Lambda_{T-1}^N(j_T, j_{T-1}) \dots \Lambda_0^N(j_1, j_0) \eqsp.
\end{equation}
Thus, the FFBS estimator \eqref{eq:forward-filtering-backward-smoothing} of the fixed-interval smoothing distribution may be written as the conditional expectation
\begin{equation*} \label{eq:EspCond}
    \post[hat]{0:T}{T}[h] = \CPE{h\left(\epart{0}{N,J_0}, \dots, \epart{T}{N,J_T} \right)}{\mcf{T}{N}} \eqsp,
\end{equation*}
where $h$ is a bounded measurable function on $\Xset^{T+1}$. We may therefore construct an unbiased estimator of the FFBS estimator given by
\begin{equation} \label{eq:FFBSi:estimator}
    \widetilde{\phi}_{0:T|T}^N[h] = N^{-1} \sum_{\ell = 1}^N h \left( \epart{0}{N,J_0^\ell}, \dots, \epart{T}{N,J_T^\ell} \right)\eqsp,
\end{equation}
where $\{\chunk{J}{0}{T}^\ell\}_{\ell=1}^N$ are $N$ paths drawn independently given $\mcf{T}{N}$ according to \eqref{eq:distribution-non-homogeneous} and where $h$ is a bounded measurable function on $\Xset^{T+1}$.
This practical estimator was introduced in \cite{godsill:doucet:west:2004} (Algorithm~1, p.~158). An implementation of this estimator whose complexity grows linearly in $N$ is introduced in \cite{douc:garivier:moulines:olsson:2010}.

\section{Non-asymptotic deviation inequalities}
\label{sec:results}
In this Section, the $\rmL_q$-mean error bounds and exponential deviation inequalities of the FFBS and FFBSi algorithms are established for additive functionals of the form \eqref{eq:addfuncbis}. Our results are established under the following assumptions.

\begin{hypA}
\label{assum:mixing}
\begin{enumerate}[(i)]
\item \label{assum:mixing:m}There exists $\left( \sigma_{-},\sigma_{+}  \right) \in (0,\infty)^2$ such that $ \sigma_{-}< \sigma_{+}$ and for any $(x,x^{\prime}) \in \Xset^2$, $ \sigma_{-}\leq m(x,x^{\prime})\leq  \sigma_{+}$ and we set $\rho \eqdef 1 - \sigma_-/\sigma_+$.
\item \label{assum:mixing:int}There exists $c_{-}\in\mathbb{R}_+^*$ such that $\int\chi(\rmd x)g_{0}(x)\geq c_{-}$ and for any $t\in\mathbb{N}^{*},$ $\inf_{x\in\Xset}\int M(x,\rmd x^{\prime})g_{t}(x^{\prime})\geq c_{-}$.
\end{enumerate}
\end{hypA}

\begin{hypA}
\label{assum:boundmodel}
\begin{enumerate}[(i)]
\item For all $t \geq 0$ and  all $x\in \Xset$, $g_{t}(x) >0$.
\item $\underset{t\geq 0}{\sup}|g_{t}|_{\infty} < \infty$.
\end{enumerate}
\end{hypA}

\begin{hypA}
\label{assum:boundalgo}
$\underset{t\geq 1}{\sup}|\vartheta_t|_{\infty} < \infty$, $\underset{t\geq 0}{\sup}|p_t|_{\infty} < \infty$ and $\underset{t\geq 0}{\sup}|\ewght{t}{}|_{\infty} < \infty$ where
    \begin{equation*}
    \omega_0(x) \eqdef \dfrac{d\chi}{d\rho_0}(x)g_{0}(x)
    , \quad
    \omega_t(x,x^{\prime}) \eqdef \dfrac{m(x,x^{\prime})g_{t}(x')}{\vartheta_t(x)p_t(x,x^{\prime})},\forall t \geq 1\eqsp.
    \end{equation*}
\end{hypA}
Assumptions A\ref{assum:mixing} and A\ref{assum:boundmodel} give bounds for the model and assumption A\ref{assum:boundalgo} for quantities related to the algorithm. A\ref{assum:mixing}\eqref{assum:mixing:m}, referred to as the strong mixing condition, is crucial to derive time-uniform exponential deviation inequalities and a time-uniform bound of the variance of the marginal smoothing distribution (see \cite{delmoral:guionnet:2001} and \cite{douc:garivier:moulines:olsson:2010}).
For all function $h$ from a space $\rmE$ to $\mathbb{R}$, $\osc(h)$ is defined by:
\[
\osc(h)\eqdef\sup_{(z,z^{\prime})\in\rmE^2} |h(z)-h(z^{\prime})|\eqsp.
\]

\begin{theorem}\label{Th:MainResult}
Assume A\ref{assum:mixing}--\ref{assum:boundalgo}. For all $q \geq 2$, there exists a constant $C$ (depending only on $q$, $\sigma_-$, $\sigma_+$, $c_-$, $\underset{t\geq 1}{\sup}|\vartheta_t|_{\infty}$ and $\underset{t\geq 0}{\sup}|\ewght{t}{}|_{\infty} $) such that for any $T<\infty$, any integer $r$ and any bounded and measurable functions $\{h_{s}\}_{s=r}^{T}$,
\begin{equation*}
\normL{q}{\phi^N_{0:T|T}\left[\addfunc{T,\size}\right]-\phi_{0:T|T}\left[\addfunc{T,\size}\right]}\leq \frac{C}{\sqrt{N}}\Upsilon^{N}_{\size,T}\left(\sum_{s=\size}^{T}\osc(h_{s})^{2}\right)^{1/2}\eqsp,
\end{equation*}
where $\addfunc{T,\size}$ is defined by \eqref{eq:addfuncbis}, $\phi^N_{0:T|T}$ is defined by \eqref{eq:forward-filtering-backward-smoothing} and where
\[
\Upsilon^{N}_{\size,T} \eqdef \sqrt{\size+1}\left(\sqrt{1+\size}\wedge\sqrt{T-\size+1} + \frac{\sqrt{1+\size}\sqrt{T-\size+1}}{\sqrt{N}}\right)\eqsp.
\]
Similarly,
\begin{equation*}
\normL{q}{\widetilde{\phi}^N_{0:T|T}\left[\addfunc{T,\size}\right]-\phi_{0:T|T}\left[\addfunc{T,\size}\right]}\leq \frac{C}{\sqrt{N}}\Upsilon^{N}_{\size,T}\left(\sum_{s=\size}^{T}\osc(h_{s})^{2}\right)^{1/2}\eqsp,
\end{equation*}
where $\widetilde{\phi}_{0:T|T}^N$ is defined by \eqref{eq:FFBSi:estimator}.
\end{theorem}

\begin{remark}
In the particular cases where $\size = 0$ and $\size =T$, $\Upsilon^{N}_{0,T} = 1 + \sqrt{T+1/N}$ and $\Upsilon^{N}_{T,T} = \sqrt{T+1}(1 + \sqrt{T+1/N})$. Then, Theorem~\ref{Th:MainResult} gives
\[
\normL{q}{\phi^N_{0:T|T}\left[\addfunc{T,0}\right]-\phi_{0:T|T}\left[\addfunc{T,0}\right]}\leq C\frac{\left(\sum_{s=0}^{T}\osc(h_{s})^{2}\right)^{1/2}}{\sqrt{N}}\left(1 + \sqrt{\frac{T+1}{N}}\right)\eqsp,
\]
and
\[
\normL{q}{\phi^N_{0:T|T}\left[\addfunc{T,T}\right]-\phi_{0:T|T}\left[\addfunc{T,T}\right]}\leq C\sqrt{\frac{T+1}{N}}\left(1 + \sqrt{\frac{T+1}{N}}\right)\osc(h_{T})^{2}\eqsp.
\]
As stated in Section~\ref{sec:intro}, theses bounds improve the results given in \cite{delmoral:doucet:singh:2010} for the FFBS estimator.
\end{remark}

\begin{remark}
The dependence on $1/\sqrt{N}$ is hardly surprising. Under the stated
strong mixing condition, it is known that the $\rmL_{q}$-norm of the
marginal smoothing estimator $\phi^N_{t-\size:t|T}[h]$, $t \in \{
\size, \dots, T\}$ is uniformly bounded in time by
$\normL{q}{\phi^N_{t-\size:t|T}[h]} \leq C \osc(h) N^{-1/2}$ (where
$C$ depends only on $q$, $\sigma_-$, $\sigma_+$, $c_-$,
$\underset{t\geq 1}{\sup}|\vartheta_t|_{\infty}$ and $\underset{t\geq
0}{\sup}|\ewght{t}{}|_{\infty} $). The dependence in $\sqrt{T}$
instead of $T$ reflects the forgetting property of the filter and the
backward smoother. As for $\size \leq s < t \leq T$, the estimators
$\phi^N_{s-\size:s|T}[h_s]$ and $\phi^N_{t-\size:t|T}[h_t]$ become
asymptotically independent as $(t-s)$ gets large, the $\rmL_{q}$-norm
of the sum $\sum_{t=r}^T \phi^N_{t-\size:t|T}[h_t]$ scales as the sum
of a mixing sequence (see \cite{davidson:1997}).
\end{remark}

\begin{remark}
It is easy to see that the scaling in $\sqrt{T/N}$ cannot in general
be improved. Assume that the kernel $m$ satisfies $m(x,x')=m(x')$ for
all $(x,x') \in \Xset \times \Xset$. In this case, for any $t \in
\{0,\dots,T\}$, the filtering distribution is
\begin{equation*}
\phi_t[h_t] = \dfrac{\int m(x)g_{t}(x)h_t(x) \rmd x}{\int m(x)g_{t}(x)
\rmd x}\eqsp,
\end{equation*}
and the backward kernel is the identity kernel. Hence, the
fixed-interval smoothing distribution coincides with the filtering
distribution. If we assume that we apply the bootstrap filter for
which $p_s(x,x')=m(x')$ and $\vartheta_s(x)=1$, the estimators $\{
\phi_{t|T}^N[h_t]\}_{t\in\{0,\dots,T\}}$ are independent
random variables corresponding to importance sampling estimators. It
is easily seen that
\begin{equation*}
\normL{q}{\sum_{t=0}^T \phi_t^N[h_t] - \phi_t[h_t]} \leq C \underset{0\leq t\leq T}{\max}\left\{\osc(h_{t})\right\}
\sqrt{\frac{T}{N}}\eqsp.
\end{equation*}
\end{remark}

\begin{remark}
The independent case also clearly illustrates why the path-space
methods are sub-optimal (see also \cite{briers:doucet:maskell:2010}
for a discussion). When applied to the independent case
(for all $(x,x')\in \Xset \times \Xset$, $m(x,x')=m(x')$ and
$p_s(x,x')=m(x')$), the asymptotic variance of the path-space
estimators is given in \cite{delmoral:2004} by
\begin{align*}
\Gamma_{0:T|T}[S_{T,0}]&\\
& \hspace{-1cm}\eqdef\sum_{t=0}^{T-1}\frac{m(g_{T}^{2})}{m(g_{T})^{2}}\frac{m(g_{t}[h_{t}-\phi_{t}(h_{t})]^{2})}{m(g_{t})}+\frac{m(g_{T}^{2}[h_{T}-\phi_{T}(h_{T})]^{2})}{m(g_{T})^{2}}\\
&\hspace{-.8cm}+\sum_{t=1}^{T-1}\left\{\sum_{s=0}^{t-1}\frac{m(g_{t}^{2})}{m(g_{t})^{2}}\frac{m(g_{s}[h_{s}-\phi_{s}(h_{s})]^{2})}{m(g_{s})}+\frac{m(g_{t}^{2}[h_{t}-\phi_{t}(h_{t})]^{2})}{m(g_{t})^{2}}\right\}\\
&\hspace{-.8cm}+\frac{\chi(g_{0}^{2}[h_{0}-\phi_{0}(h_{0})]^{2})}{\chi(g_{0})^{2}}\eqsp.
\end{align*}
The asymptotic variance thus increases as $T^2$ and hence, under the
stated assumptions, the variance of the path-space methods is of order
$T^2/N$. It is believed (and proved in some specific scenarios) that
the same scaling holds for path-space methods for non-degenerated
Markov kernel (the result has been formally established for strongly
mixing kernel under the assumption that $\sigma_-/\sigma_+$ is
sufficiently close to $1$).
\end{remark}
We provide below a brief outline of the main steps of the proofs (a detailed proof is given in Section~\ref{sec:proof:theorem:mainresult}). Following \cite{douc:garivier:moulines:olsson:2010}, the proofs rely on a decomposition of the smoothing error. For all $0\leq t \leq T$ and all bounded and measurable function $h$ on $\Xset^{T+1}$ define the kernel $\rmL_{t,T}:\Xset^{t+1}\times \sigmaX^{\otimes T+1}\rightarrow [0,1]$ by
\begin{equation*}\label{eq:ldroit}
\rmL_{t,T}h(x_{0:t})\eqdef \int \prod_{u=t+1}^TM(x_{u-1},\rmd x_u)g_{u}(x_u)h(x_{0:T})\eqsp.
\end{equation*}
The fixed-interval smoothing distribution may then be expressed, for all bounded and measurable function $h$ on $\Xset^{T+1}$, by
\begin{equation*}
\phi_{0:T|T}[h] = \frac{\phi_{0:t|t}\left[\rmL_{t,T}h\right]}{\phi_{0:t|t}\left[\rmL_{t,T}\1\right]}\eqsp,
\end{equation*}
and this suggests to decompose the smoothing error as follows
\begin{align}\label{eq:def-error}
\errorh &\eqdef \phi^N_{0:T|T}\left[h\right]-\phi_{0:T|T}\left[h\right]\\
&= \sum_{t=0}^T\frac{\phi_{0:t|t}^N\left[\rmL_{t,T}h\right]}{\phi_{0:t|t}^N\left[\rmL_{t,T}\1\right]}-\frac{\phi_{0:t-1|t-1}^N\left[\rmL_{t-1,T}h\right]}{\phi_{0:t-1|t-1}^N\left[\rmL_{t-1,T}\1\right]}\nonumber\eqsp,
\end{align}
where we used the convention
\begin{equation*}
\frac{\phi_{0:-1|-1}^N\left[\rmL_{-1,T}h\right]}{\phi_{0:-1|-1}^N\left[\rmL_{-1,T}\1\right]} = \frac{\phi_0\left[\rmL_{0,T}h\right]}{\phi_0\left[\rmL_{0,T}\1\right]} = \phi_{0:T|T}[h]\eqsp.
\end{equation*}
Furthermore, for all $0\leq t\leq T$,
\begin{align*}
\phi_{0:t|t}^N\left[\rmL_{t,T}h\right] &=\int\phi_{0:t|t}^N(\rmd x_{0:t})\rmL_{t,T}h(x_{0:t})\\
&=\int\phi_{t}^N(\rmd x_{t})\rmB_{\phi_{t-1}^N}(x_t,\rmd x_{t-1})\cdots\rmB_{\phi_{0}^N}(x_1,\rmd x_{0})\rmL_{t,T}h(x_{0:t})\\
&=\int \phi_{t}^N(\rmd x_t)\mathcal{L}_{t,T}^Nh(x_t)\eqsp,
\end{align*}
where $\mathcal{L}_{t,T}^N$ and $\mathcal{L}_{t,T}$ are two kernels on $\Xset \times \sigmaX^{\otimes (T+1)}$ defined for all $x_t\in\Xset$ by
\begin{align}
\mathcal{L}_{t,T}h(x_t) &\eqdef \int\rmB_{\phi_{t-1}}(x_t,\rmd x_{t-1})\cdots\rmB_{\phi_{0}}(x_1,\rmd x_{0})\rmL_{t,T}h(x_{0:t})\label{eq:defL}\\
\mathcal{L}_{t,T}^Nh(x_t) &\eqdef \int\rmB_{\phi_{t-1}^N}(x_t,\rmd x_{t-1})\cdots\rmB_{\phi_{0}^N}(x_1,\rmd x_{0})\rmL_{t,T}h(x_{0:t})\label{eq:defLN}\eqsp.
\end{align}
For all $1\leq t \leq T$ we can write
\begin{multline*}
\frac{\phi_{0:t|t}^N[\rmL_{t,T}h]}{\phi_{0:t|t}^N[\rmL_{t,T}\1]}-\frac{\phi_{0:t-1|t-1}^N[\rmL_{t-1,T}h]}{\phi_{0:t-1|t-1}^N[\rmL_{t-1,T}\1]}
=\frac{\phi_{t}^N[\mathcal{L}_{t,T}^Nh]}{\phi_{t}^N[\mathcal{L}_{t,T}^N\1]} - \frac{\phi_{t-1}^N[\mathcal{L}_{t-1,T}^Nh]}{\phi_{t-1}^N[\mathcal{L}_{t-1,T}^N\1]}\\
=\frac{1}{\phi_{t}^N[\mathcal{L}_{t,T}^N\1]}\left(\phi_{t}^N[\mathcal{L}_{t,T}^Nh] - \frac{\phi_{t-1}^N[\mathcal{L}_{t-1,T}^Nh]}{\phi_{t-1}^N[\mathcal{L}_{t-1,T}^N\1]}\phi_{t}^N[\mathcal{L}_{t,T}^N\1]\right)\eqsp,
\end{multline*}
and then,
\begin{equation}\label{Eq:Err}
\errorh = \sum_{t=0}^{T}\frac{N^{-1}\sum_{\ell=1}^{N}\ewght{t}{N,\ell} G_{t,T}^Nh(\epart{t}{N,\ell})}{N^{-1}\sum_{\ell=1}^{N}\ewght{t}{N,\ell} \mathcal{L}_{t,T}\1(\epart{t}{N,\ell})}\eqsp,
\end{equation}
with  $G_{t,T}^N$ is a kernel on $\Xset \times \sigmaX^{\otimes (T+1)}$ defined, for all $x_t\in\Xset$ and all bounded and measurable function $h$ on $\Xset^{T+1}$, by
\begin{equation*}
\label{eq:defG}
G_{t,T}^Nh(x_t)\eqdef \mathcal{L}_{t,T}^Nh(x_t) - \frac{\phi_{t-1}^N[\mathcal{L}_{t-1,T}^Nh]}{\phi_{t-1}^N[\mathcal{L}_{t-1,T}^N\1]}\mathcal{L}_{t,T}^N\1(x_t)\eqsp,
\end{equation*}
where, by the same convention as above,
\[
G_{0,T}^Nh(x_0)\eqdef \rmL_{0,T}h(x_0) - \frac{\phi_{0}[\mathcal{L}_{0,T}h]}{\phi_{0}[\mathcal{L}_{0,T}\1]}\mathcal{L}_{0,T}\1(x_0)\eqsp.
\]

Two families of random variables $\left\{C_{t,T}^N(f)\right\}_{t=0}^{T}$ and $\left\{D_{t,T}^N(f)\right\}_{t=0}^{T}$ are now introduced to transform (\ref{Eq:Err}) into a suitable decomposition to compute an upper bound for the $\rmL_q$-mean error. As shown in Lemma~\ref{Lem:Upperbounds}, the random variables $\{\ewght{t}{N,\ell} G_{t,T}^Nf(\epart{t}{N,\ell})\}_{\ell=1}^{N}$ are centered given $\mathcal{F}_{t-1}^N$. The idea is to replace $N^{-1}\sum_{\ell=1}^{N}\ewght{t}{N,\ell} \mathcal{L}_{t,T}\1(\epart{t}{N,\ell})$ in \eqref{Eq:Err} by its conditional expectation given $\mathcal{F}_{t-1}^N$ to get a martingale difference. This conditional expectation is computed using the following intermediate result. For any measurable function $h$ on $\Xset$ and any $t\in\{0,\dots,T\}$,
\begin{equation}
\label{eq:whExpect}
\CExp{\ewght{t}{N,1}h(\epart{t}{N,1})}{\mathcal{F}_{t-1}^N} = \frac{\phi_{t-1}^N\left[Mg_th\right]}{\phi_{t-1}^N[\vartheta_t]}\eqsp.
\end{equation}
Indeed,
\begin{align}
&\CExp{\ewght{t}{N,1}h(\epart{t}{N,1})}{\mathcal{F}_{t-1}^N} \nonumber\\
&= \CExp{\frac{\m(\epart{t-1}{N,I_t^{N,1}},\epart{t}{N,1}) g_t(\epart{t}{N,1})}{\adjfunc{t}{t}{\epart{t-1}{N,I_t^{N,1}}} \kiss{t}{t}(\epart{t-1}{N,I_t^{N,1}},\epart{t}{N,1})}h(\epart{t}{N,1})}{\mathcal{F}_{t-1}^N} \nonumber\\
&= \left(\sum_{i=1}^N\ewght{t-1}{N,i} \adjfunc{t}{t}{\epart{t-1}{N,i}}\right)^{-1}\sum_{i=1}^N\int\ewght{t-1}{N,i} \adjfunc{t}{t}{\epart{t-1}{N,i}} \kiss{t}{t}(\epart{t-1}{N,i},x)\frac{M(\epart{t-1}{N,i},\rmd x) g_t(x)}{\adjfunc{t}{t}{\epart{t-1}{N,i}} \kiss{t}{t}(\epart{t-1}{N,i},x)}h(x) \nonumber\\
&= \left(\sum_{i=1}^N\ewght{t-1}{N,i} \adjfunc{t}{t}{\epart{t-1}{N,i}}\right)^{-1}\sum_{i=1}^N\int\ewght{t-1}{N,i}M(\epart{t-1}{N,i},\rmd x) g_t(x)h(x) \nonumber\\
&=\frac{\phi_{t-1}^N\left[Mg_th\right]}{\phi_{t-1}^N[\vartheta_t]}\eqsp. \nonumber
\end{align}
This result, applied with the function $h = \mathcal{L}_{t,T}\1$, yields
\begin{equation*}
\mathbb{E}\left[\left.\ewght{t}{N,1} \mathcal{L}_{t,T}\1(\epart{t}{N,1})\right|\mathcal{F}_{t-1}^N\right] = \frac{\phi_{t-1}^N\left[Mg_t \mathcal{L}_{t,T}\1\right]}{\phi_{t-1}^N[\vartheta_t]} = \frac{\phi_{t-1}^N\left[\mathcal{L}_{t-1,T}\1\right]}{\phi_{t-1}^N[\vartheta_t]}\eqsp.
\end{equation*}

 For any $0 \leq t \leq T$, define for all bounded and measurable function $h$ on $\Xset^{T+1}$,

\begin{align}
D_{t,T}^N(h) &\eqdef \mathbb{E}\left[\left.\ewght{t}{N,1} \frac{\mathcal{L}_{t,T}\1(\epart{t}{N,1})}{|\mathcal{L}_{t,T}\1|_{\infty}}\right|\mathcal{F}_{t-1}^N\right]^{-1}N^{-1}\sum_{\ell=1}^{N}\ewght{t}{N,\ell} \frac{G_{t,T}^Nh(\epart{t}{N,\ell})}{|\mathcal{L}_{t,T}\1|_{\infty}}\label{eq:defD}\\
&= \frac{\phi_{t-1}^N[\vartheta_t]}{\phi_{t-1}^N\left[\frac{\mathcal{L}_{t-1,T}\1}{|\mathcal{L}_{t,T}\1|_{\infty}}\right]}N^{-1}\sum_{\ell=1}^{N}\ewght{t}{N,\ell} \frac{G_{t,T}^Nh(\epart{t}{N,\ell})}{|\mathcal{L}_{t,T}\1|_{\infty}}\eqsp,\nonumber
\end{align}

\begin{multline}\label{eq:defC}
\hspace{0.7cm}C_{t,T}^N(h) \eqdef \left[\frac{1}{N^{-1}\sum_{i=1}^{N}\ewght{t}{N,i} \frac{\mathcal{L}_{t,T}\1(\epart{t}{N,i})}{|\mathcal{L}_{t,T}\1|_{\infty}}}-\frac{\phi_{t-1}^N[\vartheta_t]}{\phi_{t-1}^N\left[\frac{\mathcal{L}_{t-1,T}\1}{|\mathcal{L}_{t,T}\1|_{\infty}}\right]}\right] \\
\times N^{-1}\sum_{\ell=1}^{N}\ewght{t}{N,\ell} \frac{G_{t,T}^Nh(\epart{t}{N,\ell})}{|\mathcal{L}_{t,T}\1|_{\infty}}
\eqsp.
\end{multline}
Using these notations, \eqref{Eq:Err} can be rewritten as follows:
\begin{equation}\label{eq:decomp}
\errorh = \sum_{t=0}^{T}D_{t,T}^N(h) + \sum_{t=0}^{T}C_{t,T}^N(h)\eqsp.
\end{equation}
For any $q \geq 2$, the derivation of the upper bound relies on the triangle inequality:
\begin{equation*}\label{Eq:NormL}
\normL{q}{ \error } \leq \normL{q}{\sum_{t=0}^{T}D_{t,T}^N(\addfunc{T,\size})} + \sum_{t=0}^{T}\normL{q}{C_{t,T}^N(\addfunc{T,\size})}\eqsp,
\end{equation*}
where $\addfunc{T,\size}$ is defined in \eqref{eq:addfuncbis}. The proof for the FFBS estimator $\phi^N_{0:T|T}$ is completed by using Proposition~\ref{Prop:NormD} and Proposition~\ref{Prop:NormC}.
According to  \eqref{eq:decomp}, the smoothing error can be decomposed into a sum of two terms which are considered separately. The first one is a martingale whose $\rmL_q$-mean error is upper-bounded by $\sqrt{\left(T+1\right)/N}$ as shown in Proposition \ref{Prop:NormD}. The second one is a sum of products, $\rmL_q$-norm of which being bounded by $1/N$ in Proposition \ref{Prop:NormC}.

The end of this section is devoted to the exponential deviation inequality for the error $\error$ defined by \eqref{eq:def-error}. We use the decomposition of $\error$ obtained in \eqref{eq:decomp} leading to a similar dependence on the ratio $(T+1)/N$. The martingale term $D^N_{t,T}(\addfunc{T,\size})$ is dealt with using the Azuma-Hoeffding inequality while the term $C^N_{t,T}(\addfunc{T,\size})$ needs a specific Hoeffding-type inequality for ratio of random variables.

\begin{theorem}\label{Th:ExpIneq}
Assume A\ref{assum:mixing}--\ref{assum:boundalgo}. There exists a constant $C$ (depending only on $\sigma_-$, $\sigma_+$, $\size$, $c_-$, $\underset{t\geq 1}{\sup}|\vartheta_t|_{\infty}$ and $\underset{t\geq 0}{\sup}|\ewght{t}{}|_{\infty} $)  such that for any $T<\infty$, any $N \geq 1$, any $\varepsilon > 0$, any integer $r$, and any bounded and measurable functions $\{h_{s}\}_{s=r}^{T}$,
\begin{multline*}
\mathbb{P}\left\{ \left|\phi_{0:T|T}\left[\addfunc{T,\size}\right]-\phi_{0:T|T}^N\left[\addfunc{T,\size}\right]\right| > \varepsilon \right\} \\ \leq
 2\exp\left(-\frac{CN\varepsilon^2}{\Theta_{\size,T}\sum_{s=\size}^{T}\osc(h_{s})^{2}}\right)+8\exp\left(-\frac{CN\varepsilon}{(1+\size)\sum_{s=\size}^{T}\osc(h_{s})}\right)\eqsp,
\end{multline*}
where $\addfunc{T,\size}$ is defined by \eqref{eq:addfuncbis}, $\phi^N_{0:T|T}$ is defined by \eqref{eq:forward-filtering-backward-smoothing} and where
\begin{equation}
\label{eq:defTheta}
\Theta_{\size,T} \eqdef (1+\size)\left\{(1+\size)\wedge(T-\size+1)\right\}\eqsp.
\end{equation}
Similarly,
\begin{multline*}
\mathbb{P}\left\{ \left|\phi_{0:T|T}\left[\addfunc{T,\size}\right]-\widetilde{\phi}_{0:T|T}^N\left[\addfunc{T,\size}\right]\right| > \varepsilon \right\} \\ \leq 4\exp\left(-\frac{CN\varepsilon^2}{\Theta_{\size,T}\sum_{s=\size}^{T}\osc(h_{s})^{2}}\right)+8\exp\left(-\frac{CN\varepsilon}{(1+\size)\sum_{s=\size}^{T}\osc(h_{s})}\right)\eqsp,
\end{multline*}
where $\widetilde{\phi}_{0:T|T}^N$ is defined by \eqref{eq:FFBSi:estimator}.
\end{theorem}

\section{Monte-Carlo Experiments}
\label{sec:MCexp}
In this section, the performance of the FFBSi algorithm is evaluated through simulations and compared to the path-space method.
\subsection{Linear gaussian model}
Let us consider the following model:
\begin{equation*}
\begin{cases}
X_{t+1} &= \phi X_t + \sigma_uU_t\eqsp,\\
Y_t &= X_t +  \sigma_vV_t\eqsp,
\end{cases}
\end{equation*}
where $X_0$ is a zero-mean random variable with variance $\frac{\sigma_u^2}{1-\phi^2}$, $\left\{U_t\right\}_{t\geq 0}$ and $\left\{V_t\right\}_{t\geq 0}$ are two sequences of independent and identically distributed standard gaussian random variables (independent from $X_0$). The parameters $\left(\phi,\sigma_u,\sigma_v\right)$ are assumed to be known. Observations were generated using $\phi = 0.9$, $\sigma_u = 0.6$ and $\sigma_v = 1$. Table~\ref{tab:LGMvar} provides the empirical variance of the estimation of the unnormalized smoothed additive functional $\mathcal{I}_T \eqdef \sum_{t=0}^{T}\CExp{X_t}{Y_{0:T}}$ given by the path-space and the FFBSi methods over $250$ independent Monte Carlo experiments.
\begin{table}[h!]
\begin{center}
\caption{Empirical variance for different values of $T$ and $N$.}
\label{tab:LGMvar}
\begin{minipage}[t]{\linewidth}
\begin{tabular}{|c|ccccccccc|}
\multicolumn{3}{l}{\textbf{Path-space}}\\
\hline\hline
 \backslashbox{$T$}{$N$}            &   300     & 500   & 750   & 1000 & 1500  & 5000    & 10000  & 15000 & 20000  \\
\hline
  300                               &	137.8   & 119.4	& 63.7	& 46.1 & 36.2  & 12.8	 & 7.1	  & 3.8	  & 3.0\\
500                 &290.0&	215.3&	192.5&	161.9&	80.3   &   30.1	&  14.9	&11.3&	7.4\\
750                 &474.9&	394.5	&332.9&	250.5&	206.8  &   71.0	&35.6	& 24.4	&21.7\\
1000               &	673.7&	593.2&	505.1	&483.2&	326.4  &   116.4	&70.8	&37.9	&34.6\\
1500               &	1274.6&	1279.7&	916.7&	804.7&	655.1    &   233.9	&163.1	&89.7	&80.0\\
\hline\hline
\end{tabular}
\begin{tabular}{|c|ccccc|}
\multicolumn{3}{l}{\textbf{FFBSi}}\\
\hline\hline
 \backslashbox{$T$}{$N$}           &   300 & 500 & 750 & 1000 & 1500    \\
 \hline
  300              &	5.1&	3.1&	2.3	&1.4&	1.0\\
500               &	9.7	&5.1	&3.7	&2.6&	2.2\\
750               	&11.2	&7.1&	4.9	&3.7	&2.6\\
1000               &	16.5	&10.5	&6.7	&5.1	&3.4\\
1500               &\phantom{0}25.6\phantom{0}	&\phantom{0}14.1\phantom{0}&	\phantom{0}7.8\phantom{0}	 &\phantom{0}6.8\phantom{0}&	\phantom{0}5.1\phantom{0}\\
\hline \hline
\end{tabular}
\end{minipage}
\end{center}
\end{table}
 We display in Figure~\ref{fig:LGMvar} the empirical variance for different values of $N$ as a function of $T$ for both estimators. These estimates are represented by dots and a linear regression (resp. quadratic regression) is also provided for the FFBSi algorithm (resp. for the path-space method).
 \begin{figure}[h!]
\centering
 \begin{minipage}[b]{1\linewidth}
  \centering
  \includegraphics[width=0.85\linewidth]{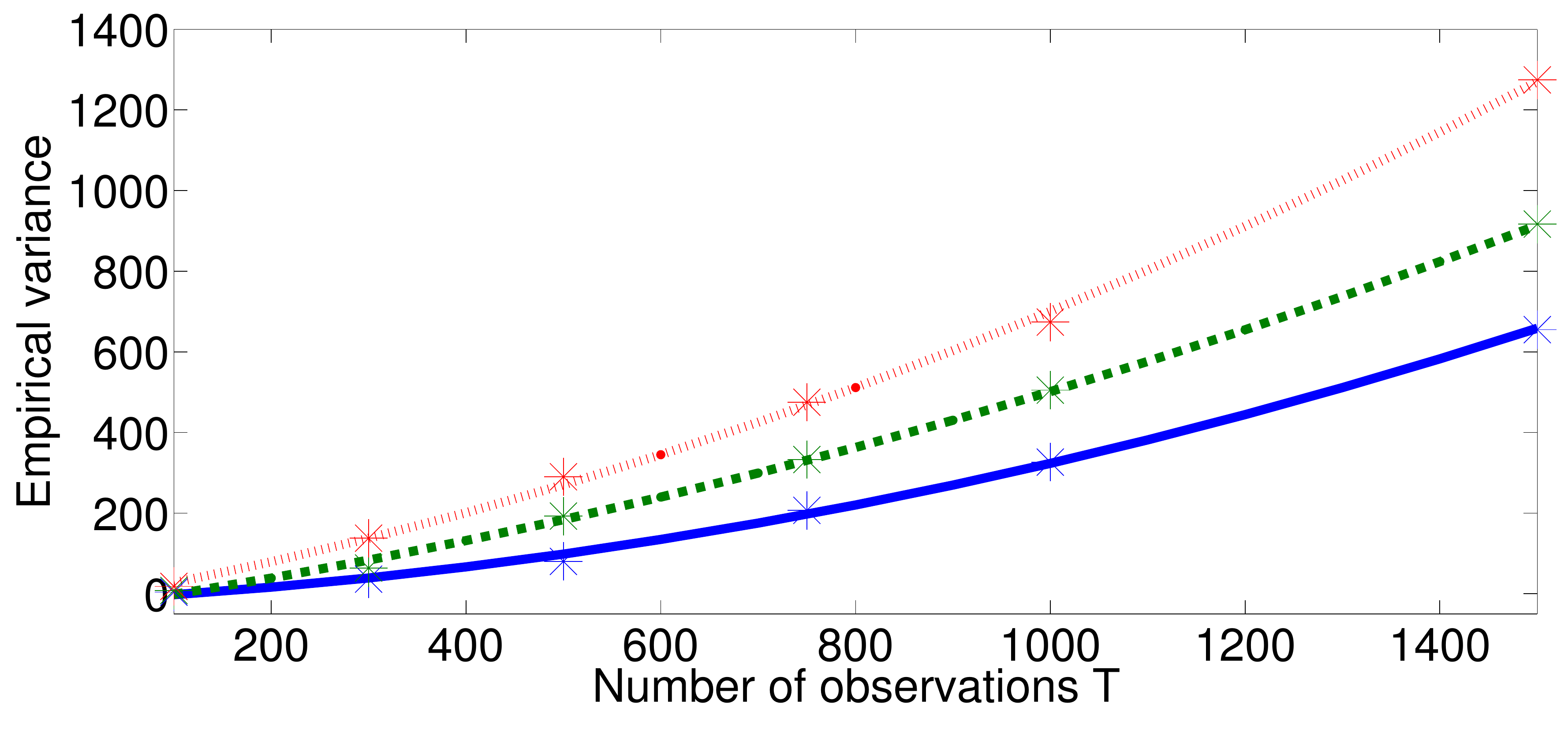}
 \end{minipage}

\vspace{0.1cm}

 \begin{minipage}[b]{1\linewidth}
  \centering
  \includegraphics[width=0.85\linewidth]{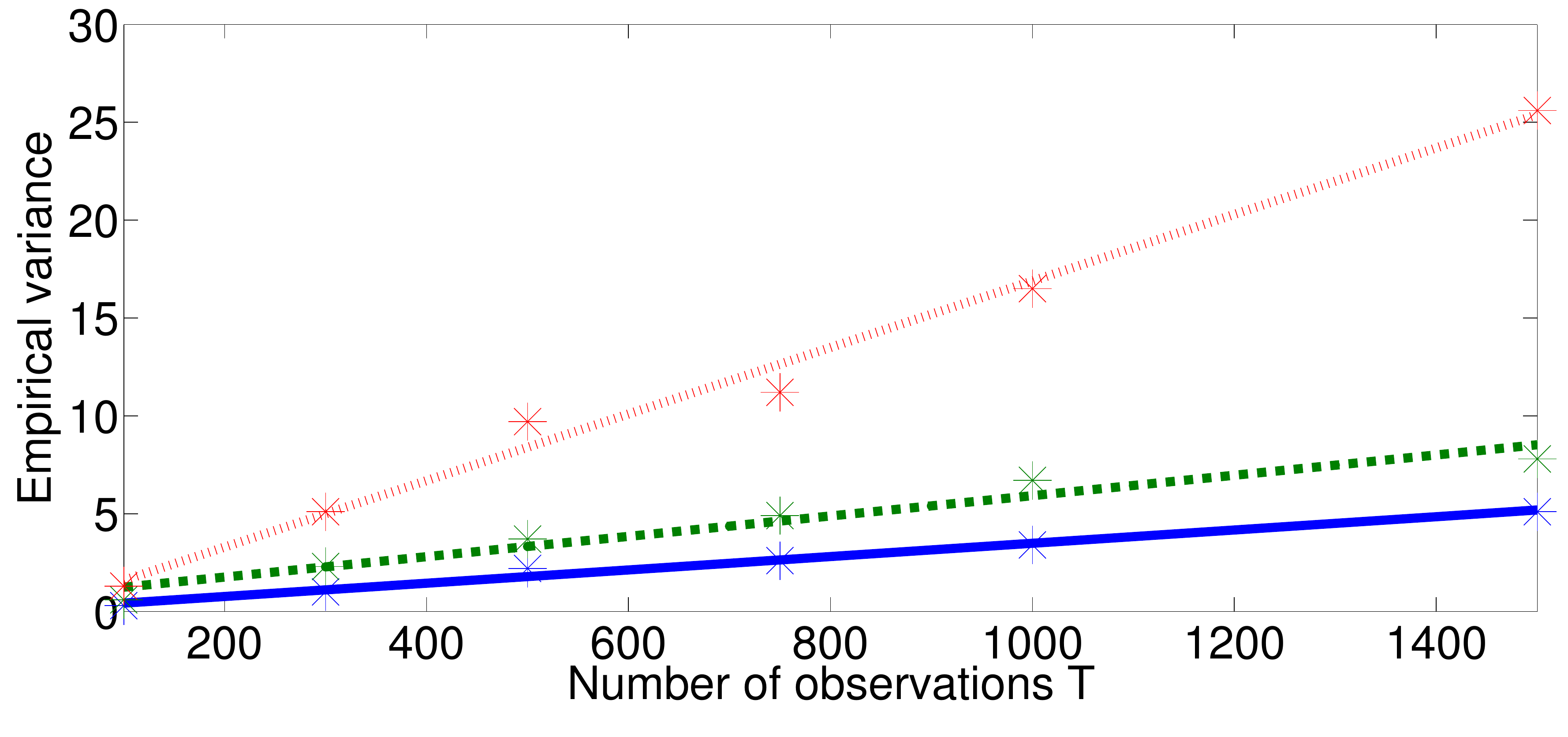}
 \end{minipage}
\caption{Empirical variance of the path-space (top) and FFBSi (bottom) for $N=300$ (dotted line), $N=750$ (dashed line) and $N=1500$ (bold line).}
\label{fig:LGMvar}
\end{figure}

In Figure \ref{fig:boxplotsmoothedLGM} the FFBSi algorithm is compared to the path-space method to compute the smoothed value of the empirical mean $(T+1)^{-1}\mathcal{I}_T$. For the purpose of comparison, this quantity is computed using the Kalman smoother. We display in Figure \ref{fig:boxplotsmoothedLGM} the box and whisker plots of the estimations obtained with $100$ independent Monte Carlo experiments. The FFBSi algorithm clearly outperforms the other method for comparable computational costs. In Table \ref{tab:LGMtime}, the mean CPU times over the $100$ runs of the two methods are given as a function of the number of particles (for $T=500$ and $T=1000$).

\begin{figure}[!h]
  \centering
  \subfloat[Time $T = 500$]{\label{fig:boxplotsmoothed500}\includegraphics[width=0.6\textwidth]{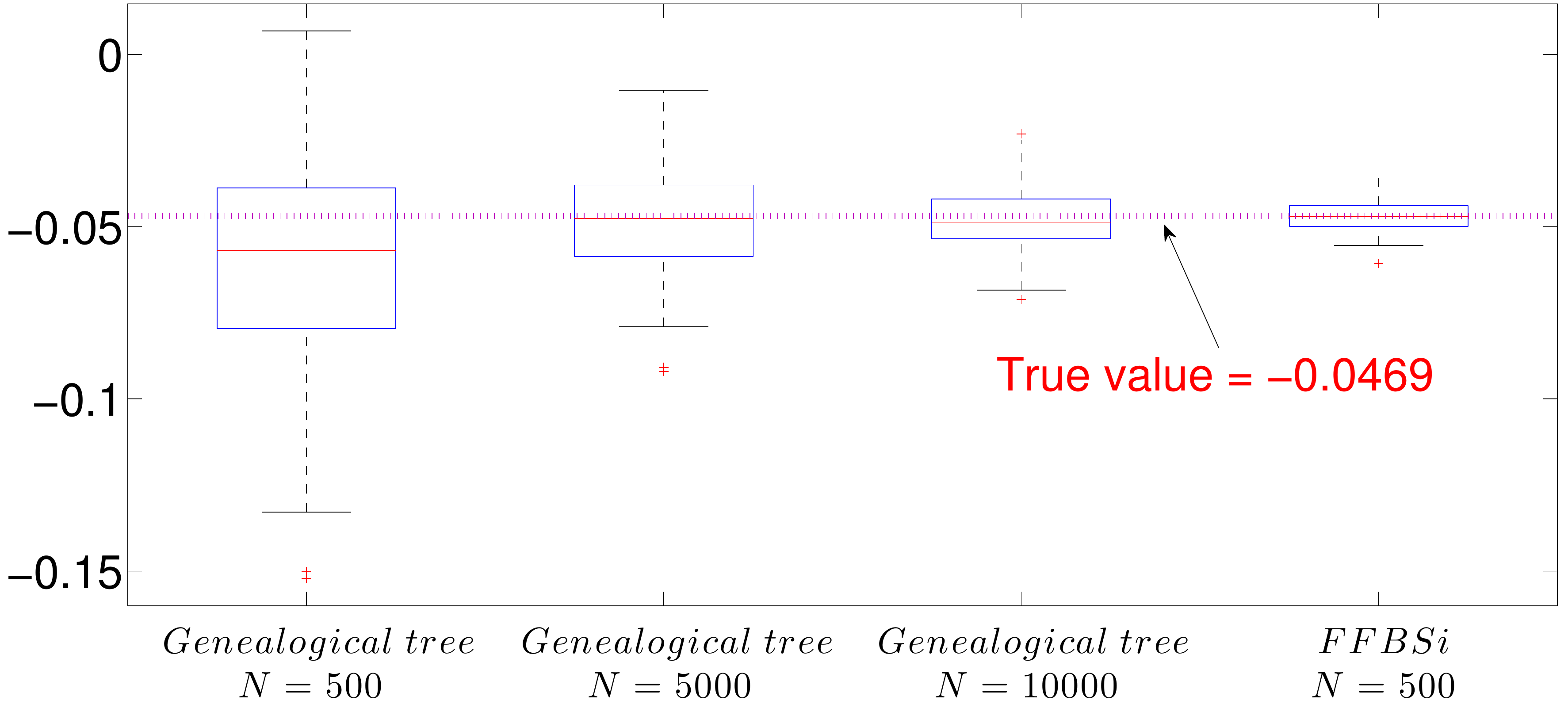}}\\
  \subfloat[Time $T = 1000$]{\label{fig:boxplotsmoothed1000}\includegraphics[width=0.6\textwidth]{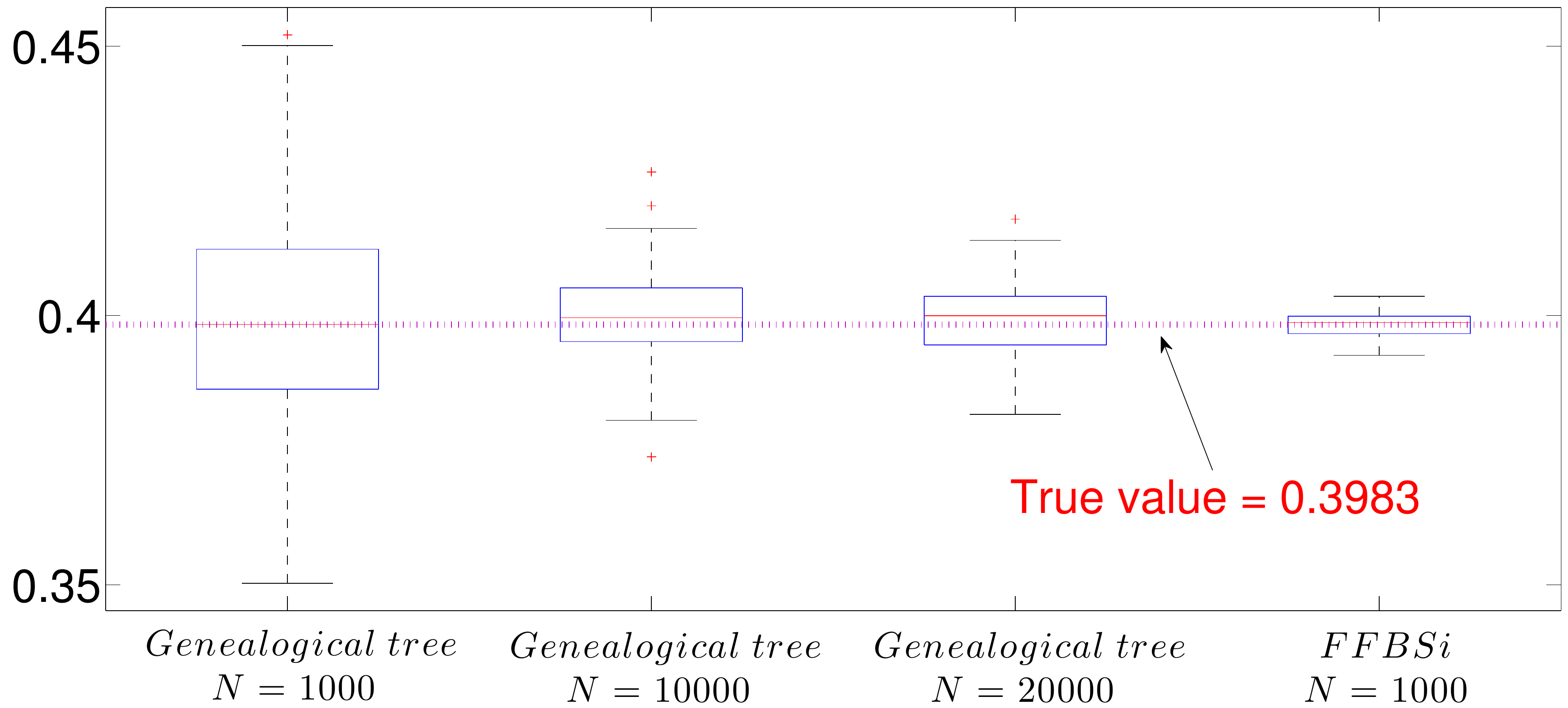}}
  \caption{Computation of smoothed additive functionals in a linear gaussian model. The variance of the estimation given by the FFBSi algorithm is the smallest one in both cases.}
  \label{fig:boxplotsmoothedLGM}
\end{figure}

\begin{table}[h!]
\centering
\caption{Average CPU time to compute the smoothed value of the empirical mean in the LGM}
\label{tab:LGMtime}
\begin{minipage}[t]{.6\linewidth}
    \begin{tabular}{cccccc}
\hline\hline
 $T=500$           &   FFBSi   &   &   \multicolumn{3}{c}{Path-space method}     \\
\\[-7.5pt] \cline{2-2} \cline{4-6}\\[-5pt]
$N$         &   500    &   &   500    &   5000  & 10000\\
CPU time (s)&   4.87      &   &   0.24     &   2.47      & 4.65\\
\hline \hline
\end{tabular}
\end{minipage}

\vspace{0.3cm}

\begin{minipage}[t]{.6\linewidth}
    \begin{tabular}{cccccc}
\hline\hline
$T=1000$            &   FFBSi   &   &   \multicolumn{3}{c}{Path-space method}     \\
\\[-7.5pt] \cline{2-2} \cline{4-6}\\[-5pt]
$N$         &   1000    &   &   1000    &   10000  & 20000\\
CPU time (s)&   16.5      &   &   0.9     &   8.5      & 17.2\\
\hline \hline
\end{tabular}
\end{minipage}
\end{table}

\subsection{Stochastic Volatility Model}
Stochastic volatility models (SVM) have been introduced to provide
better ways of modeling financial time series data than ARCH/GARCH
models (\cite{hull:white:1987}). We consider the elementary SVM model
introduced by \cite{hull:white:1987}:
\begin{equation*}
\begin{cases}
X_{t+1} = \phi X_t + \sigma U_{t+1}\eqsp, \\
Y_t = \beta \rme^{\frac{X_t}{2}} V_t\eqsp,
\end{cases}
\end{equation*}
where $X_0$ is a zero-mean random variable with variance $\frac{\sigma_u^2}{1-\phi^2}$, $\left\{U_t\right\}_{t\geq 0}$ and $\left\{V_t\right\}_{t\geq 0}$ are two sequences of independent and identically distributed standard gaussian random variables (independent from $X_0$).
This model was used to generate simulated data with parameters
$(\phi=0.3,\sigma=0.5,\beta=1)$ assumed to be known in the following
experiments. The empirical variance of the estimation of  $\mathcal{I}_T$ given by the path-space and the FFBSi methods over $250$ independent Monte Carlo experiments is
displayed in Table \ref{tab:SVMvar}.
\begin{table}[h!]
\centering
\caption{Empirical variance for different values of $T$ and $N$ in the SVM.}
\label{tab:SVMvar}
\begin{minipage}[t]{\linewidth}
\begin{tabular}{|c|ccccccccc|}
\multicolumn{3}{l}{\textbf{Path-space method}}\\
\hline\hline
 \backslashbox{$T$}{$N$}           &   300 & 500 & 750 & 1000 & 1500  & 5000    &   10000   &   15000   &   20000  \\
\hline
 300               &   52.7&   33.7    &22.0   &17.8&  12.3 &   3.8	&2.0	&1.4	&1.2\\
500                 &116.3&     84.8&   64.8&   53.5&   30.7    &   11.4	&6.8	&4.1	&2.8\\
750                 &184.7&     187.6   &134.2& 120.0&  65.8    &   29.1	&12.8	&7.3	&7.7\\
1000               &    307.7&  240.4&  244.7   &182.8& 133.2   &   43.6	&24.5	&15.6	&11.6\\
1500               &    512.1&  487.5&  445.5&  359.9&  249.5   &   90.9	&52.0	&32.6	&29.3\\
\hline\hline
\end{tabular}
\begin{tabular}{|c|ccccc|}
\multicolumn{3}{l}{\textbf{FFBSi}}\\
\hline\hline
 \backslashbox{$T$}{$N$}           &   300 & 500 & 750 & 1000 & 1500    \\
 \hline
 300              &    1.2&    0.6&    0.5     &0.4&   0.2\\
500               &     2.1     &1.2    &0.8    &0.6&   0.4\\
750                     &3.7    &1.8&   1.4     &0.9    &0.6\\
1000               &    4.0     &2.7    &1.8    &1.3    &0.9\\
1500               &\phantom{0}7.3\phantom{0} &\phantom{0}3.8\phantom{0}&   \phantom{0}3.1\phantom{0}     &\phantom{0}1.6\phantom{0}&   \phantom{0}1.4\phantom{0}\\
\hline \hline
\end{tabular}
\end{minipage}
\end{table}
 We display in Figure \ref{fig:SVMvar} the empirical variance for
different values of $N$ as a function of $T$ for both estimators.
 \begin{figure}[h!]
\centering
 \begin{minipage}[b]{1\linewidth}
 \centering
 \includegraphics[width=0.85\linewidth]{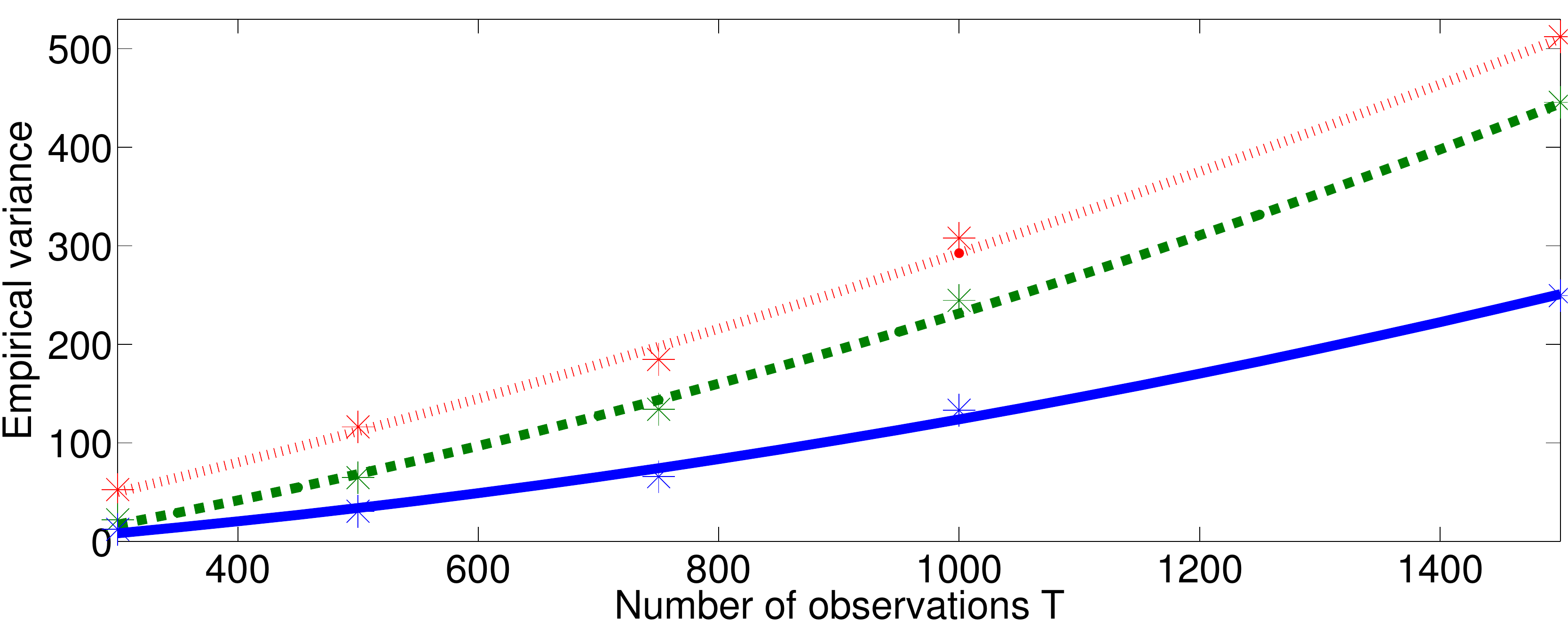}
 \end{minipage}

\vspace{0.1cm}

 \begin{minipage}[b]{1\linewidth}
 \centering
 \includegraphics[width=0.85\linewidth]{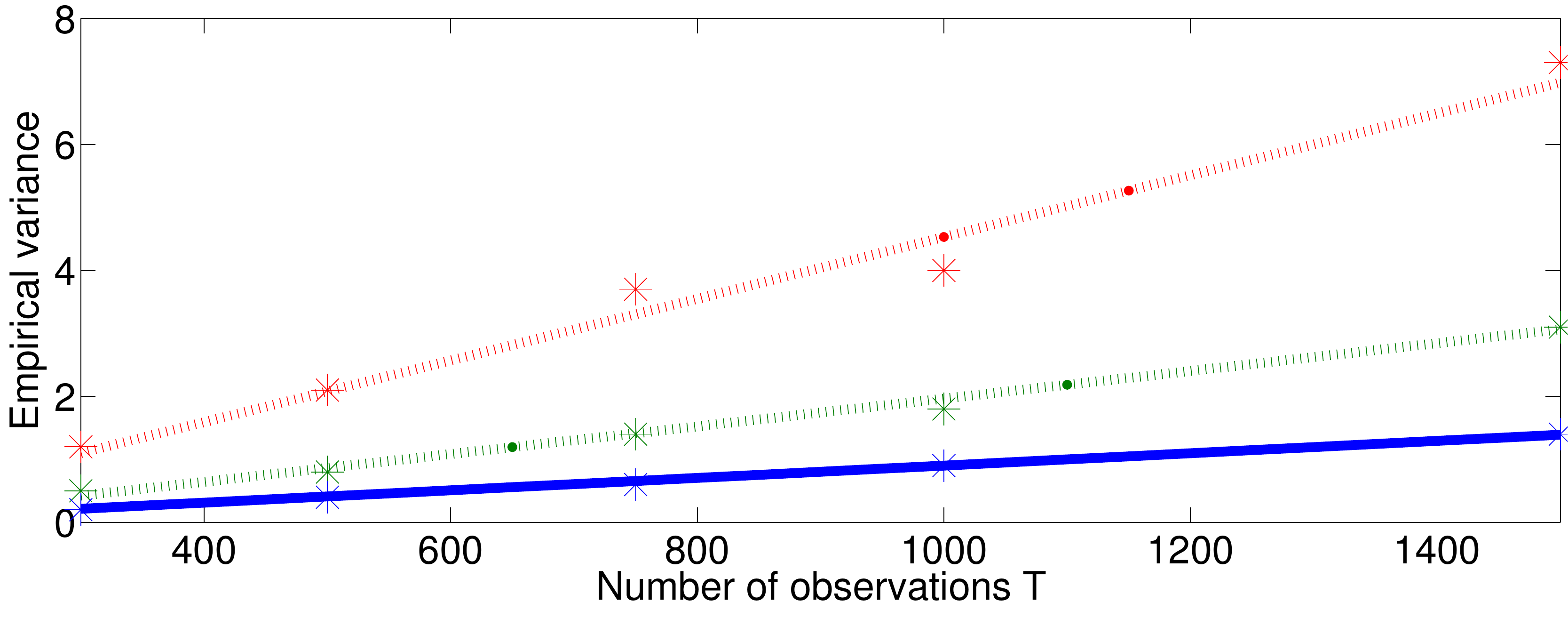}
 \end{minipage}
\caption{Empirical variance of the path-space (top) and FFBSi
(bottom) for $N=300$ (dotted line), $N=750$ (dashed line) and $N=1500$
(bold line) in the SVM.}
\label{fig:SVMvar}
\end{figure}

\section{Proof of Theorem~\ref{Th:MainResult}}
\label{sec:proof:theorem:mainresult}
We preface the proof of Proposition \ref{Prop:NormD} by the following Lemma:
\begin{lemma}
\label{Lem:Upperbounds}
Under assumptions A\ref{assum:mixing}--\ref{assum:boundalgo}, we have, for any $t\in \{0,\dots,T\}$ and any measurable function $h$ on $\Xset^{T+1}$:
\begin{enumerate}[(i)]
\item \label{Lem:Upperbounds:ind}The random variables $\displaystyle \left\{ \ewght{t}{N,\ell}\dfrac{G^N_{t,T}h(\epart{t}{N,\ell})}{|\mathcal{L}_{t,T}\1|_\infty}\right\}_{\ell=1}^{N}$ are, for all $N\in\mathbb{N}$:
    \begin{enumerate}
    \item \label{Lem:Upperbounds:ind:claim1} conditionally independent and identically distributed given $\mathcal{F}_{t-1}^N$\eqsp,
    \item \label{Lem:Upperbounds:ind:claim2}centered conditionally to $\mathcal{F}_{t-1}^N$\eqsp.
    \end{enumerate}
    where $G^N_{t,T}h$ is defined in \eqref{eq:defG} and $\mathcal{L}_{t,T}^{N}$ is defined in \eqref{eq:defLN}.
\item \label{Lem:Upperbounds:boundup} For any integers $r$, $t$ and $N$:
\begin{equation}\label{eq:bound-G}
\left| \dfrac{G^N_{t,T}\addfunc{T,\size}(\epart{t}{N,\ell})}{|\mathcal{L}_{t,T}\1|_\infty}\right|\leq \sum_{s=r}^T \rho^{\max(t-s,s-r-t,0)}\osc(h_s)\eqsp,
\end{equation}
where $\addfunc{T,\size}$ and $\rho$ are respectively defined in \eqref{eq:addfuncbis} and in A\ref{assum:mixing}\eqref{assum:mixing:m}.
\item \label{Lem:Upperbounds:boundlow}
For all $x\in\Xset$, $\displaystyle \dfrac{\mathcal{L}_{t,T}\1(x)}{|\mathcal{L}_{t,T}\1|_\infty} \geq \dfrac{\sigma_-}{\sigma_+}$ and $\displaystyle \dfrac{\mathcal{L}_{t-1,T}\1(x)}{|\mathcal{L}_{t,T}\1|_\infty} \geq c_- \dfrac{\sigma_-}{\sigma_+}$\eqsp.
\end{enumerate}
\end{lemma}

\begin{proof}
The proof of \eqref{Lem:Upperbounds:ind} is given by \GisCentered.

{\em Proof of \eqref{Lem:Upperbounds:boundup}}. Let $\Pi_{s-\size:s,T}$ be the operator which associates to any bounded and measurable function $h$ on $\Xset^{\size+1}$ the function $\Pi_{s-\size:s,T}h$ given, for any $(x_0,\dots,x_T) \in \Xset^{T+1}$, by
\begin{equation*}
\Pi_{s-\size:s,T}h(x_{0:T}) \eqdef h(x_{s-\size:s})\eqsp.
\end{equation*}
Then, we may write $\addfunc{T,\size} = \sum_{s=\size}^T \Pi_{s-\size:s,T}h_s$ and $G^N_{t,T}\addfunc{T,\size} = \sum_{s=\size}^TG^{N}_{t,T}\Pi_{s-\size:s,T}h_s$.
By \eqref{eq:defG}, we have
\begin{equation*}
\frac{G_{t,T}^N \Pi_{s-\size:s,T}h_s(x_t)}{\mathcal{L}_{t,T}^N\1(x_t)} =  \frac{\mathcal{L}_{t,T}^N \Pi_{s-\size:s,T}h_s(x_t)}{\mathcal{L}_{t,T}^N\1(x_t)} - \frac{\phi_{t-1}^N[\mathcal{L}_{t-1,T}^N \Pi_{s-\size:s,T}h_s]}{\phi_{t-1}^N[\mathcal{L}_{t-1,T}^N\1]}\eqsp,
\end{equation*}
and, following the same lines as in \GNorm,
\begin{equation*}
|G^N_{t,T}\Pi_{s-\size:s,T}h_s|_\infty \leq \rho^{s-r-t} \osc(h_s) |\mathcal{L}_{t,T}\1|_\infty \quad \mbox{if}\quad t\leq s-r\eqsp,
\end{equation*}

\begin{equation*}
|G^N_{t,T}\Pi_{s-\size:s,T}h_s|_\infty \leq \rho^{t-s} \osc(h_s) |\mathcal{L}_{t,T}\1|_\infty \quad \mbox{if}\quad t> s\eqsp,
\end{equation*}
where $\rho$ is defined in A\ref{assum:mixing}\eqref{assum:mixing:m}. Furthermore, for any $s-r<t\leq s$,
\[
|G^N_{t,T}\Pi_{s-\size:s,T}h_s|_\infty \leq  \osc(h_s) |\mathcal{L}_{t,T}\1|_\infty\eqsp,
\]
which shows \eqref{Lem:Upperbounds:boundup}.\\
{\em Proof of \eqref{Lem:Upperbounds:boundlow}}.
From the definition \eqref{eq:defL}, for all $x\in\Xset$ and all $t\in\{1,\dots,T\}$,
\begin{equation*}
\mathcal{L}_{t,T}\1(x) = \int m(x,x_{t+1})g_{t+1}(x_{t+1})\prod_{u=t+2}^{T}M(x_{u-1},\rmd x_u)g_u(x_u)\lambda(\rmd x_{t+1})\eqsp,
\end{equation*}
hence, by assumption A\ref{assum:mixing},
\begin{align*}
\left|\mathcal{L}_{t,T}\1\right|_{\infty}&\leq \sigma_+\int g_{t+1}(x_{t+1})\mathcal{L}_{t+1,T}\1(x_{t+1})\lambda(\rmd x_{t+1})\\
\mathcal{L}_{t,T}\1(x)&\geq \sigma_- \int g_{t+1}(x_{t+1})\mathcal{L}_{t+1,T}\1(x_{t+1})\lambda(\rmd x_{t+1})\eqsp,
\end{align*}
which concludes the proof of the first statement. By construction, for any $x\in\Xset$ and any $t\in\{1,\dots,T\}$,
\begin{equation*}
\label{eq:Lrec}
\mathcal{L}_{t-1,T}\1(x) = \int M(x,\rmd x^{\prime})g_t(x^{\prime})\mathcal{L}_{t,T}\1(x^{\prime})\eqsp,
\end{equation*}
and then, by assumption A\ref{assum:mixing},
\begin{equation*}\label{Eq:InfLdiff}
\dfrac{\mathcal{L}_{t-1,T}\1(x)}{|\mathcal{L}_{t,T}\1|_\infty} =\int M(x,\rmd x^{\prime})g_t(x^{\prime}) \dfrac{\mathcal{L}_{t,T}\1(x^{\prime})}{|\mathcal{L}_{t,T}\1|_\infty} \geq c_- \dfrac{\sigma_-}{\sigma_+}\eqsp.
\end{equation*}
\end{proof}

\begin{proposition}\label{Prop:NormD}
Assume A\ref{assum:mixing}--\ref{assum:boundalgo}. For all $q \geq2$, there exists a constant $C$ (depending only on $q$, $\sigma_-$, $\sigma_+$, $c_-$, $\underset{t\geq 1}{\sup}|\vartheta_t|_{\infty}$ and $\underset{t\geq 0}{\sup}|\ewght{t}{}|_{\infty} $) such that for any $T<\infty$, any integer $r$ and any bounded and measurable functions $\{h_{s}\}_{s=r}^{T}$ on $\Xset^{\size+1}$,
\begin{equation}\label{Eq:NormD}
\normL{q}{\sum_{t=0}^{T}D_{t,T}^N(\addfunc{T,\size})} \leq \frac{C}{\sqrt{N}}\sqrt{1+\size}\left(\sqrt{1+\size}\wedge\sqrt{T-\size+1}\right)\left(\sum_{s=\size}^{T}\osc(h_{s})^{2}\right)^{1/2}\eqsp,
\end{equation}
where $D_{t,T}^N$ is defined in \eqref{eq:defD}.
\end{proposition}

\begin{proof}
Since $\left\{D_{t,T}^N(\addfunc{T,\size})\right\}_{0\leq t\leq T}$ is a  is a forward martingale difference and $q\geq 2$, Burkholder's inequality (see \cite[Theorem 2.10, page 23]{hall:heyde:1980}) states the existence of a constant $C$ depending only on $q$ such that:
\begin{equation*}
\normLpw{q}{\sum_{t=0}^{T}D_{t,T}^N(\addfunc{T,\size})} \leq C\normLpw{\frac{q}{2}}{\sum_{t=0}^{T}D_{t,T}^N(\addfunc{T,\size})^2}\eqsp.
\end{equation*}
Moreover, by application of the last statement of Lemma \ref{Lem:Upperbounds}\eqref{Lem:Upperbounds:boundlow},
\begin{equation*}
\frac{\phi_{t-1}^N[\vartheta_t]}{\phi_{t-1}^N\left[\frac{\mathcal{L}_{t-1,T}\1}{|\mathcal{L}_{t,T}\1|_{\infty}}\right]} \leq  \frac{\sigma_{+} \sup_{t\geq 0}|\vartheta_t|_{\infty}}{\sigma_{-}c_{-}}\eqsp,
\end{equation*} and thus,
\begin{equation*}
\normLpw{\frac{q}{2}}{\sum_{t=0}^{T}D_{t,T}^N(\addfunc{T,\size})^2} \leq   \left(\frac{\sigma_{+} \sup_{t\geq 0}|\vartheta_t|_{\infty}}{\sigma_{-}c_{-}}\right)^q\normLpw{\frac{q}{2}}{\sum_{t=0}^T\left(N^{-1}\sum_{\ell=1}^{N}a_{t,T}^{N,\ell}\right)^2}\eqsp,
\end{equation*}
where $a_{t,T}^{N,\ell}\eqdef \ewght{t}{N,\ell} \frac{G_{t,T}^N\addfunc{T,\size}(\epart{t}{N,\ell})}{|\mathcal{L}_{t,T}\1|_{\infty}}$. By the Minkowski inequality,
\begin{equation}
\label{eq:jensenD}
\normL{q}{\sum_{t=0}^{T}D_{t,T}^N(\addfunc{T,\size})} \leq C \left\{\sum_{t=0}^T\left(\mathbb{E}\left[\left|N^{-1}\sum_{\ell=1}^{N}a_{t,T}^{N,\ell}\right|^q\right]\right)^{2/q}\right\}^{1/2}\eqsp.
\end{equation}
Since for any $t\geq 0$ the random variables $\displaystyle \left\{ a_{t,T}^{N,\ell}\right\}_{\ell=1}^{N}$ are conditionally independent and centered conditionally to $\mathcal{F}_{t-1}^N$, using again the Burkholder and the Jensen inequalities we obtain
\begin{align}
\label{eq:boundA}
\CExp{\left|\sum_{\ell=1}^{N}a_{t,T}^{N,\ell}\right|^q}{\mathcal{F}_{t-1}^N}&\leq CN^{q/2-1}\sum_{\ell=1}^N \CExp{\left|a_{t,T}^{N,\ell}\right|^q}{\mathcal{F}_{t-1}^N} \nonumber \\
&\leq C\left[\sum_{s=r}^T \rho^{\max(t-s,s-r-t,0)}\osc(h_s)\right]^qN^{q/2}\eqsp,
\end{align}
where the last inequality comes from \eqref{eq:bound-G}. Finally, by \eqref{eq:jensenD} and \eqref{eq:boundA} we get
\begin{align*}
\normL{q}{\sum_{t=0}^{T}D_{t,T}^N(\addfunc{T,\size})} &\leq C N^{-1/2} \left\{\sum_{t=0}^{T}\left(\sum_{s=r}^T \rho^{\max(t-s,s-r-t,0)}\osc(h_s)\right)^{2}\right\}^{1/2}\eqsp.
\end{align*}
By the Holder inequality, we have
\begin{align*}
\sum_{s=r}^T \rho^{\max(t-s,s-r-t,0)}\osc(h_s)&\\
&\hspace{-2.7cm}\le \left(\sum_{s=r}^T \rho^{\max(t-s,s-r-t,0)}\right)^{1/2}\times\left(\sum_{s=r}^T \rho^{\max(t-s,s-r-t,0)}\osc(h_s)^{2}\right)^{1/2}\\
&\hspace{-2.7cm}\le C\sqrt{1+\size} \left(\sum_{s=r}^T \rho^{\max(t-s,s-r-t,0)}\osc(h_s)^{2}\right)^{1/2}\eqsp,
\end{align*}
which yields
\[
\normL{q}{\sum_{t=0}^{T}D_{t,T}^N(\addfunc{T,\size})} \leq C N^{-1/2}(1+\size) \left(\sum_{s=r}^T\osc(h_s)^{2}\right)^{1/2}\eqsp.
\]
We obtain similarly
\[
\normL{q}{\sum_{t=0}^{T}D_{t,T}^N(\addfunc{T,\size})} \leq C N^{-1/2}(1+\size)^{1/2} \sum_{s=r}^T\osc(h_s)\eqsp,
\]
which concludes the proof.
\end{proof}

\begin{proposition}
\label{Prop:NormC}
Assume A\ref{assum:mixing}--\ref{assum:boundalgo}. For all $q \geq 2$, there exists a constant $C$ (depending only on $q$, $\sigma_-$, $\sigma_+$, $c_-$, $\underset{t\geq 1}{\sup}|\vartheta_t|_{\infty}$ and $\underset{t\geq 0}{\sup}|\ewght{t}{}|_{\infty} $) such that for any $T<+\infty$, any $0 \leq t \leq T$, any integer $r$, and any bounded and measurable functions $\{h_{s}\}_{s=r}^{T}$ on $\Xset^{\size+1}$,
\begin{equation}\label{Eq:NormC}
\normL{q}{C_{t,T}^N(\addfunc{T,\size})} \leq  \dfrac{C}{N}\sum_{s=r}^T \rho^{\max(t-s,s-r-t,0)}\osc(h_s)\eqsp,
\end{equation}
where $C_{t,T}^N$ is defined in \eqref{eq:defC}.
\end{proposition}

\begin{proof}
According to \eqref{eq:defC}, $C_{t,T}^N(\addfunc{T,\size})$ can be written
\begin{equation}\label{eq:decC}
C_{t,T}^N(\addfunc{T,\size}) = U_{t,T}^NV_{t,T}^NW_{t,T}^N\eqsp,
\end{equation}
where
\begin{align*}
U_{t,T}^N &= \dfrac{ N^{-1} \sum_{\ell=1}^N \ewght{t}{N,\ell} \frac{G^N_{t,T}\addfunc{T,\size}(\epart{t}{N,\ell})}{|\mathcal{L}_{t,T}\1|_\infty}}{ N^{-1} \Omega_t^N}\eqsp,\\
V_{t,T}^N &=N^{-1}\sum_{\ell=1}^N\left(\mathbb{E}\left[\left.\ewght{t}{N,1} \frac{\mathcal{L}_{t,T}\1(\epart{t}{N,1})}{|\mathcal{L}_{t,T}\1|_\infty}\right|\mathcal{F}_{t-1}\right]- \ewght{t}{N,\ell} \frac{\mathcal{L}_{t,T}\1(\epart{t}{N,\ell})}{|\mathcal{L}_{t,T}\1|_\infty}\right)\eqsp,\\
W_{t,T}^N &= \frac{ N^{-1} \Omega_t^N }{\mathbb{E}\left[\left.\ewght{t}{N,1} \frac{\mathcal{L}_{t,T}\1(\epart{t}{N,1})}{|\mathcal{L}_{t,T}\1|_\infty}\right|\mathcal{F}_{t-1}\right]N^{-1}\sum_{\ell=1}^N \ewght{t}{N,\ell} \frac{\mathcal{L}_{t,T}\1(\epart{t}{N,\ell})}{|\mathcal{L}_{t,T}\1|_\infty}}\eqsp,
\end{align*}
and where $\Omega_t^N$ is defined by \eqref{eq:defOmega}.
Using the last statement of Lemma \ref{Lem:Upperbounds}, we get the following bound:
\begin{align*}
\mathbb{E}\left[\left.\ewght{t}{N,1} \frac{\mathcal{L}_{t,T}\1(\epart{t}{N,1})}{|\mathcal{L}_{t,T}\1|_\infty}\right|\mathcal{F}_{t-1}\right] & =  \frac{\phi_{t-1}^N\left[\mathcal{L}_{t-1,T}\1/|\mathcal{L}_{t,T}\1|_\infty\right]}{\phi_{t-1}^N[\vartheta_t]}\\
&\geq \frac{c_-\sigma_-}{\left|\vartheta_t\right|_{\infty}\sigma_+}\\
\frac{\mathcal{L}_{t,T}\1(\epart{t}{N,\ell})}{|\mathcal{L}_{t,T}\1|_\infty}&\geq \frac{\sigma_-}{\sigma_+}\eqsp,
\end{align*}
which implies
\begin{equation}\label{Eq:SupIneqForW}
\left|W_{t,T}^N\right| \leq \left( \dfrac{\sigma_+}{\sigma_-}\right)^2 \dfrac{|\vartheta_t|_\infty}{c_-}\eqsp.
\end{equation}
Then, $\left|C_{t,T}^N(\addfunc{T,\size})\right| \leq C\left|U_{t,T}^N\right|\left|V_{t,T}^N\right|$ and we can use the decomposition
\[
U_{t,T}^NV_{t,T}^N = V_{t,T}^N \left[\dfrac{ N^{-1} \sum_{\ell=1}^N a_{t,T}^{N,\ell}}{\mathbb{E}\left[\left.\widetilde{\Omega}_t^N\right|\mathcal{F}_{t-1}\right]} + \dfrac{ N^{-1} \sum_{\ell=1}^Na_{t,T}^{N,\ell}}{\widetilde{\Omega}_t^N\mathbb{E}\left[\left.\widetilde{\Omega}_t^N\right|\mathcal{F}_{t-1}\right]}\left(\mathbb{E}\left[\left.\widetilde{\Omega}_t^N\right|\mathcal{F}_{t-1}\right]-\widetilde{\Omega}_t^N\right)\right]\eqsp,
\]
where $a_{t,T}^{N,\ell}\eqdef \ewght{t}{N,\ell} \frac{G^N_{t,T}\addfunc{T,\size}(\epart{t}{N,\ell})}{|\mathcal{L}_{t,T}\1|_\infty}$ and $\widetilde{\Omega}_t^N\eqdef N^{-1}\Omega_t^N$. By \eqref{eq:whExpect}, $\CExp{\ewght{t}{N,1}}{\mathcal{F}_{t-1}^N} = \frac{\phi_{t-1}^N\left[Mg_t\right]}{\phi_{t-1}^N[\vartheta_t]}$ and then, by A\ref{assum:mixing}\eqref{assum:mixing:int}, A\ref{assum:boundalgo} and \eqref{eq:bound-G},
\[
\dfrac{1}{\mathbb{E}\left[\left.\widetilde{\Omega}_t^N\right|\mathcal{F}_{t-1}\right]} \le \frac{|\vartheta_t|_\infty}{c_{-}}\quad\mbox{and}\quad\dfrac{ N^{-1} \sum_{\ell=1}^Na_{t,T}^{N,\ell}}{\widetilde{\Omega}_t^N\mathbb{E}\left[\left.\widetilde{\Omega}_t^N\right|\mathcal{F}_{t-1}\right]}\le C\frac{|\vartheta_t|_\infty}{c_{-}}\sum_{s=r}^T \rho^{\max(t-s,s-r-t,0)}\osc(h_s)\eqsp.
\]
Therefore, $\left|C_{t,T}^N(\addfunc{T,\size})\right| \leq C\left(C_{t,T}^{1,N}+\sum_{s=r}^T \rho^{\max(t-s,s-r-t,0)}\osc(h_s)C_{t,T}^{2,N}\right)$ where
\[
C_{t,T}^{1,N}\eqdef V_{t,T}^N \cdot N^{-1} \sum_{\ell=1}^N a_{t,T}^{N,\ell}\quad\mbox{and}\quad C_{t,T}^{2,N}\eqdef V_{t,T}^N\left|\mathbb{E}\left[\left.\widetilde{\Omega}_t^N\right|\mathcal{F}_{t-1}\right]-\widetilde{\Omega}_t^N\right|\eqsp.
\]

The random variables $\left\{\ewght{t}{N,\ell} \frac{\mathcal{L}_{t,T}\1(\epart{t}{N,\ell})}{|\mathcal{L}_{t,T}\1|_\infty}\right\}_{\ell=1}^{N}$ being bounded and conditionally independent given $\mathcal{F}_{t-1}^N$, following the same steps as in the proof of Proposition \ref{Prop:NormD}, there exists a constant $C$ (depending only on $q$, $\sigma_-$, $\sigma_+$, $c_-$ and $\underset{t\geq 0}{\sup}|\ewght{t}{}|_{\infty} $) such that $\normL{2q}{V_{t,T}^N}\leq CN^{-1/2}$. Similarly
\[
\normL{2q}{N^{-1} \sum_{\ell=1}^N a_{t,T}^{N,\ell}}\le C\frac{\sum_{s=r}^T \rho^{\max(t-s,s-r-t,0)}\osc(h_s)}{N^{1/2}}\eqsp,
\]
and
\[\normL{2q}{\mathbb{E}\left[\left.\widetilde{\Omega}_t^N\right|\mathcal{F}_{t-1}\right]-\widetilde{\Omega}_t^N}\le \frac{C}{N^{1/2}}\eqsp.
\]
The Cauchy-Schwarz inequality concludes the proof of \eqref{Eq:NormC}.
\end{proof}

The proof of Theorem \ref{Th:MainResult} is now concluded for the FFBS estimator $\phi^N_{0:T|T}\left[\addfunc{T,\size}\right]$ and we can proceed to the proof for the FFBSi estimator. We preface the proof  of Theorem \ref{Th:MainResult} for the FFBSi estimator $\widetilde \phi^N_{0:T|T}$ by the following Lemma. We first define the backward filtration $\left\{\mathcal{G}_{t,T}^N\right\}_{t=0}^{T+1}$ by
\begin{equation*}\label{eq:BackwardFiltration}
\begin{cases}
\mathcal{G}_{T+1,T}^N \eqdef \mathcal{F}_T^N   \eqsp,  \\
\mathcal{G}_{t,T}^N \eqdef \mathcal{F}_T^N \vee \sigma\left\{ J_u^\ell, 1\leq\ell\leq N, t \leq u \leq T \right\},\quad \forall \eqsp t \in\{0,\dots,T\}\eqsp.
\end{cases}
\end{equation*}
\begin{lemma}\label{Lem:BackInitForget}
Assume A\ref{assum:mixing}--\ref{assum:boundalgo}. Let $\ell\in\{1,\dots,N\}$ and $T<+\infty$. For any bounded measurable function $h$ on $\Xset^{r+1}$ we have,
\begin{enumerate}[(i)]
\item \label{Lem:BackInitForget:claim1}for all $u,t$ such that $r \leq t \leq u \leq T$,
\begin{equation*}
\left|\mathbb{E}\left[ h\left(\epart{t-r:t}{N,J_{t-r:t}^\ell}\right)\middle|\mathcal{G}_{u,T}^N\right] - \mathbb{E}\left[ h\left(\epart{t-r:t}{N,J_{t-r:t}^\ell}\right)\middle|\mathcal{G}_{u+1,T}^N\right]\right| \leq \rho^{u-t} \osc(h)\eqsp,
\end{equation*}
where $\rho$ is defined in  A\ref{assum:mixing}\eqref{assum:mixing:m}.
\item \label{Lem:BackInitForget:claim2} for all $u,t$ such that $t-r \leq u \leq t-1 \leq T$,
\begin{equation*}
\left|\mathbb{E}\left[ h\left(\epart{t-r:t}{N,J_{t-r:t}^\ell}\right)\middle|\mathcal{G}_{u,T}^N\right] - \mathbb{E}\left[ h\left(\epart{t-r:t}{N,J_{t-r:t}^\ell}\right)\middle|\mathcal{G}_{u+1,T}^N\right]\right| \leq \osc(h)\eqsp.
\end{equation*}
\end{enumerate}
\end{lemma}
\begin{proof}
According to Section \ref{subsec:FFBSi}, for all $\ell\in\{1,\dots,N\}$, $\{ J_{u}^{N,\ell} \}_{u = 0}^T$ is an inhomogeneous Markov chain evolving backward in time with backward kernel $\{\Lambda^N_{u}\}_{u = 0}^{T-1}$. For any $r \leq t \leq u \leq T$, we have
\begin{multline*}
\mathbb{E}\left[ h\left(\epart{t-r:t}{N,J_{t-r:t}^{N,\ell}}\right)\middle|\mathcal{G}_{u,T}^N\right] - \mathbb{E}\left[ h\left(\epart{t-r:t}{N,J_{t-r:t}^{N,\ell}}\right)\middle|\mathcal{G}_{u+1,T}^N\right] \\
= \sum_{j_{t:u}} \left[ \delta_{J_u^{N,\ell}}(j_u) - \left( \Lambda_u(J_{u+1}^{N,\ell},j_u) \bold{1}_{u<T} + \frac{\ewght{T}{N,j_u}}{\Omega_u} \bold{1}_{u=T}\right) \right]\\
\times\prod_{\ell=u}^{t+1} \Lambda_{\ell-1}^N(j_\ell,j_{\ell-1})\sum_{j_{t-r:t-1}}\prod_{\ell=t}^{t-r+1}\Lambda_{\ell-1}^N(j_\ell,j_{\ell-1})h\left(\epart{t-r:t}{N,j_{t-r:t}}\right)\eqsp.
\end{multline*}
The RHS of this equation is the difference between two expectations started with two different initial distributions.
Under A\ref{assum:mixing}\eqref{assum:mixing:m}, the backward kernel satisfies the uniform Doeblin condition,
\begin{equation*}
\forall (i,j)\in\{1,\dots,N\}^2\quad \Lambda_s^N(i,j) \geq \frac{\sigma_-}{\sigma_+} \frac{\omega_s^i}{\Omega_s^N}\eqsp,
\end{equation*}
and the proof is completed by the exponential forgetting of the backward kernel (see \cite{cappe:moulines:ryden:2005,delmoral:guionnet:2001}). The proof of \eqref{Lem:BackInitForget:claim2} follows exactly the same lines.
\end{proof}
To compute an upper-bound for the $\rmL_q$-mean error of the FFBSi algorithm, we may define the difference between the FFBS and the FFBSi estimators:
\begin{equation}\label{eq:errorsmall}
\errorsmall\left[\addfunc{T,\size}\right] = \widetilde{\phi}^N_{0:T|T}\left[\addfunc{T,\size}\right] - \phi^N_{0:T|T}\left[\addfunc{T,\size}\right]\eqsp.
\end{equation}
\begin{proof}[Proof of Theorem \ref{Th:MainResult} for the FFBSi estimator]
The difference between the FFBS and the FFBSi estimators, $\errorsmall$, defined in \eqref{eq:errorsmall}, can be written
\begin{align*}
\errorsmall\left[\addfunc{T,\size}\right] &= \frac{1}{N} \sum_{\ell=1}^N \sum_{t=r}^T h_t\left(\epart{t-r:t}{N,J_{t-r:t}^{N,\ell}}\right) - \mathbb{E}\left[ h_t\left(\epart{t-r:t}{N,J_{t-r:t}^{N,1}}\right)\middle|\mathcal{F}_T^N\right] \\
 &= \frac{1}{N} \sum_{\ell=1}^N \sum_{t=r}^T \sum_{u=t-r}^T \mathbb{E}\left[ h_t\left(\epart{t-r:t}{N,J_{t-r:t}^{N,\ell}}\right)\middle|\mathcal{G}_{u,T}^N\right] - \mathbb{E}\left[ h_t\left(\epart{t-r:t}{N,J_{t-r:t}^{N,\ell}}\right)\middle|\mathcal{G}_{u+1,T}^N\right]\\
 &= \frac{1}{N} \sum_{\ell=1}^N \sum_{u=0}^T \zeta_u^{N,\ell} \eqsp,
\end{align*}
where
\begin{align*}
\zeta_u^{N,\ell}&\eqdef \sum_{t=r}^{(u+r)\wedge T} \mathbb{E}\left[ h_t\left(\epart{t-r:t}{N,J_{t-r:t}^{N,\ell}}\right)\middle|\mathcal{G}_{u,T}^N\right] - \mathbb{E}\left[ h_t\left(\epart{t-r:t}{N,J_{t-r:t}^{N,\ell}}\right)\middle|\mathcal{G}_{u+1,T}^N\right]\eqsp.
\end{align*}
For all $\ell \in \{1,\dots,N\}$ and all $u\in\{0,\dots,T\}$, the random variable $\zeta_u^{N,\ell}$ is $\mathcal{G}_{u,T}^N$-measurable and $\mathbb{E}\left[ \zeta_u^{N,\ell} \middle| \mathcal{G}_{u+1,T}^N\right]=0$ so that $\zeta_u^{N,\ell}$ can be seen as the increment of a backward martingale. Hence, since $q\geq 2$, using the Burkholder inequality (see \cite[Theorem 2.10, page 23]{hall:heyde:1980}), there exists a constant $C$ (depending only on $q$, $\sigma_-$, $\sigma_+$, $c_-$, $\underset{t\geq 1}{\sup}|\vartheta_t|_{\infty}$ and $\underset{t\geq 0}{\sup}|\ewght{t}{}|_{\infty} $) such that:
\begin{equation}\label{eq:backwardburk}
\normL{q}{\errorsmall\left[\addfunc{T,\size}\right]}\le C\left\{\sum_{u=0}^{T}\normLpw{q}{N^{-1}\sum_{\ell=1}^N \zeta_u^{N,\ell}}^{2/q}\right\}^{1/2}\eqsp.
\end{equation}
Then, since the random variables $\{\zeta_u^{N,\ell}\}_{\ell=1}^{N}$ are conditionally independent and centered conditionally to $\mathcal{G}_{u+1,T}^N$, using the Burkholder inequality once again implies:
\begin{equation}\label{eq:backwardMZ}
\CExp{\left|\sum_{\ell=1}^N \zeta_u^{N,\ell}\right|^q}{\mathcal{G}_{u+1,T}^N}\leq CN^{q/2-1}\sum_{\ell=1}^N\CExp{\left|\zeta_u^{N,\ell}\right|^q}{\mathcal{G}_{u+1,T}^N}\eqsp.
\end{equation}
Furthermore, according to Lemma~\ref{Lem:BackInitForget}\eqref{Lem:BackInitForget:claim1},
\begin{multline}\label{eq:backwardineq}
\left| \zeta_u^{N,\ell}\right| \leq \sum_{t=r}^u \left|\mathbb{E}\left[ h_t\left(\epart{t-r:t}{N,J_{t-r:t}^{N,\ell}}\right)\middle|\mathcal{G}_{u,T}^N\right] - \mathbb{E}\left[ h_t\left(\epart{t-r:t}{N,J_{t-r:t}^{N,\ell}}\right)\middle|\mathcal{G}_{u+1,T}^N\right]\right|\\
+ \sum_{t=u+1}^{(u+\size)\wedge T} \left|\mathbb{E}\left[ h_t\left(\epart{t-r:t}{N,J_{t-r:t}^{N,\ell}}\right)\middle|\mathcal{G}_{u,T}^N\right] - \mathbb{E}\left[ h_t\left(\epart{t-r:t}{N,J_{t-r:t}^{N,\ell}}\right)\middle|\mathcal{G}_{u+1,T}^N\right]\right|\\
\leq \sum_{t=r}^u \rho^{u-t} \osc(h_t) + \sum_{t=u+1}^{(u+\size)\wedge T} \osc(h_t)\eqsp.
\end{multline}
Putting \eqref{eq:backwardburk}, \eqref{eq:backwardMZ} and \eqref{eq:backwardineq} together leads to
\begin{equation*}
\normL{q}{\errorsmall\left[\addfunc{T,\size}\right]}\leq \frac{C}{\sqrt{N}}\left\{\sum_{u=0}^{T}\left(\sum_{t=\size}^{(u+\size)\wedge T}\rho^{(u-t)\vee0}\osc(h_{t})\right)^{2}\right\}^{1/2}\eqsp.
\end{equation*}
Using the Holder inequality as in the proof of Proposition~\ref{Prop:NormD} yields
\[
\normL{q}{\errorsmall\left[\addfunc{T,\size}\right]}\leq \frac{C}{\sqrt{N}}\sqrt{1+\size}\left(\sqrt{1+\size}\wedge\sqrt{T-\size+1}\right)\left(\sum_{s=\size}^{T}\osc(h_{s})^{2}\right)^{1/2}\eqsp,
\]
and the proof of Theorem \ref{Th:MainResult} for the FFBSi estimator is derived from the triangle inequality:
\begin{equation*}
\normL{q}{\phi_{0:T|T}\left(\addfunc{T,\size}\right)-\widetilde{\phi}_{0:T|T}^N\left(\addfunc{T,\size}\right)} \leq \normL{q}{\error}+\normL{q}{\errorsmall\left[\addfunc{T,\size}\right]}\eqsp,
\end{equation*}
where $\error$ is defined by \eqref{eq:def-error} and $\errorsmall\left[\addfunc{T,\size}\right]$ is defined by \eqref{eq:errorsmall}.
\end{proof}

\section{Proof of Theorem~\ref{Th:ExpIneq}}
\label{sec:proof:theorem:expineq}
We preface the proof of the Theorem by showing that the martingale term of the error $\error$ (which is defined by \eqref{eq:def-error}) satisfies an exponential deviation inequality in the following Proposition.
\begin{proposition}
\label{Prop:ExpD}
Assume A\ref{assum:mixing}--\ref{assum:boundalgo}. There exists a constant $C$ (depending only on $\sigma_-$, $\sigma_+$, $c_-$, $\underset{t\geq 1}{\sup}|\vartheta_t|_{\infty}$ and $\underset{t\geq 0}{\sup}|\ewght{t}{}|_{\infty} $)  such that for any $T<\infty$, any $N \geq 1$, any $\varepsilon > 0$, any integer $r$ and any bounded and measurable functions $\{h_{s}\}_{s=r}^{T}$ on $\Xset^{\size+1}$,
\begin{equation}\label{Eq:ExpD}
\mathbb{P}\left\{\left|\sum_{t=0}^T D_{t,T}^N(\addfunc{T,\size})\right| > \varepsilon \right\} \leq 2\exp\left(-\frac{CN\varepsilon^2}{\Theta_{\size,T}\sum_{s=\size}^{T}\osc(h_{s})^{2}}\right)\eqsp,
\end{equation}
where $D_{t,T}^N$ is defined in \eqref{eq:defD} and $\Theta_{\size,T}$ is defined by \eqref{eq:defTheta}.
\end{proposition}

\begin{proof}
According to the definition of $D_{t,T}^N(\addfunc{T,\size})$ given in \eqref{eq:defD}, we can write
\begin{equation*}
\sum_{t=0}^T D_{t,T}^N(\addfunc{T,\size}) = \sum_{k=1}^{N(T+1)} \upsilon^N_k  \eqsp,
\end{equation*}
where for all $t\in\{0,\dots,T\}$ and $\ell\in\{1,\dots,N\}$, $\upsilon^N_{Nt+\ell}$ is defined by
\begin{equation*}
\upsilon^N_{Nt+\ell} = \frac{\phi_{t-1}^N[\vartheta_t]}{\phi_{t-1}^N\left[\frac{\mathcal{L}_{t-1,T}\1}{|\mathcal{L}_{t,T}\1|_{\infty}}\right]}N^{-1}\ewght{t}{N,\ell} \frac{G_{t,T}^N\addfunc{T,\size}(\epart{t}{N,\ell})}{|\mathcal{L}_{t,T}\1|_{\infty}}    \eqsp,
\end{equation*}
and is bounded by (see \eqref{eq:bound-G})
\begin{equation*}
\left| \upsilon^N_{Nt+\ell} \right| \leq C N^{-1} \sum_{s=r}^T \rho^{\max(t-s,s-r-t,0)}\osc(h_s)   \eqsp.
\end{equation*}
Furthermore, we define the filtration $\left\{ \mathcal{H}^N_k \right\}_{k=1}^{N(T+1)}$, for all $t\in\{0,\dots,T\}$ and $\ell\in\{1,\dots,N\}$, by:
\begin{equation*}
\mathcal{H}^N_{Nt+\ell} \eqdef \mathcal{F}^N_{t-1} \vee \sigma\left\{ \left(\ewght{t}{N,i},\epart{t}{N,i}\right), 1 \leq i \leq \ell \right\} \eqsp,
\end{equation*}
with the convention $\mathcal{F}^N_{-1} = \sigma(Y_{0:T})$. Then, according to Lemma \ref{Lem:Upperbounds}, $\left\{\upsilon_k \right\}_{k=1}^{N(T+1)}$ is martingale increment for the filtration $\{\mathcal{H}_{k}^{N}\}_{k=1}^{N(T+1)}$ and the Azuma-Hoeffding inequality completes the proof.
\end{proof}

\begin{proposition}
\label{Prop:ExpC}
Assume A\ref{assum:mixing}--\ref{assum:boundalgo}. There exists a constant $C$ (depending only on $\sigma_-$, $\sigma_+$, $c_-$, $\underset{t\geq 1}{\sup}|\vartheta_t|_{\infty}$ and $\underset{t\geq 0}{\sup}|\ewght{t}{}|_{\infty} $)  such that for any $T<\infty$, any $N \geq 1$, any $\varepsilon > 0$, any integer $r$ and any bounded and measurable functions $\{h_{s}\}_{s=r}^{T}$ on $\Xset^{\size+1}$,
\begin{equation}\label{Eq:ExpC}
\mathbb{P}\left\{\left|\sum_{t=0}^T C_{t,T}^N(\addfunc{T,\size})\right| > \varepsilon \right\} \leq 8\exp\left(-\frac{CN\varepsilon}{(1+\size)\sum_{s=\size}^{T}\osc(h_{s})}\right)\eqsp.
\end{equation}
where $C_{t,T}^N(F)$ is defined in \eqref{eq:defC}.
\end{proposition}

\begin{proof}
In order to apply Lemma \ref{Lem:SumExpDev} in the appendix, we first need to find an exponential deviation inequality for $C_{t,T}^N(\addfunc{T,\size})$ which is done by using the decomposition $C_{t,T}^N(\addfunc{T,\size}) = U_{t,T}^NV_{t,T}^NW_{t,T}^N$ given in \eqref{eq:decC}. First, the ratio $U_{t,T}^N$ is dealt with through Lemma \ref{lem:inegEssentielle} in the appendix by defining
$$\left\{ \begin{array}{l@{\eqdef}l}
a_N     & N^{-1} \sum_{\ell=1}^N \ewght{t}{N,\ell} G^N_{t,T}\addfunc{T,\size}(\epart{t}{N,\ell})/|\mathcal{L}_{t,T}\1|_\infty \eqsp,\\
b_N     & N^{-1} \sum_{\ell=1}^N \ewght{t}{N,\ell}\eqsp, \\
b       & \mathbb{E}[ \omega_t^1 | \mathcal{F}_{t-1}^N] = \phi^N_{t-1} \left[ Mg_t \right] / \phi^N_{t-1}[\vartheta_t] \eqsp,\\
\beta   & c_- / |\vartheta_t|_\infty\eqsp.
\end{array}\right.$$
Assumption A\ref{assum:mixing}\eqref{assum:mixing:int} and A\ref{assum:boundalgo} shows that $b \geq \beta$ and  \eqref{eq:bound-G} shows that $|a_N/b_N| \leq C (1+\size) \maxosc{T}$. Therefore, Condition (I) of Lemma \ref{lem:inegEssentielle} is satisfied. The bounds $0<\omega_t^l\leq |\omega_t|_\infty$ and the Hoeffding inequality lead to
\begin{multline*}
\mathbb{P}[|b_N-b| \geq \varepsilon] = \mathbb{E}\left[ \mathbb{P}\left[\left. \left| N^{-1}\sum_{\ell=1}^N \left( \ewght{t}{N,\ell} - \mathbb{E}[ \ewght{t}{N,1} | \mathcal{F}_{t-1}^N] \right) \right| \geq \varepsilon \right|\mathcal{F}_{t-1}^N\right] \right]
\\ \leq 2 \exp \left( -\frac{2 N \varepsilon^2 }{ |\omega_t|_\infty^2 }\right)\eqsp,
\end{multline*}
establishing Condition (ii) in Lemma \ref{lem:inegEssentielle}. Finally, Lemma~\ref{Lem:Upperbounds}\eqref{Lem:Upperbounds:ind} and the Hoeffding inequality imply that
\begin{multline*}
\mathbb{P}\left[ |a_N  | \geq \varepsilon \right] = \mathbb{E}\left[ \mathbb{P}\left[\left. \left| N^{-1} \sum_{\ell=1}^N \ewght{t}{N,\ell} G^N_{t,T}\addfunc{T,\size}(\epart{t}{N,\ell})/|\mathcal{L}_{t,T}\1|_\infty \right| \geq \varepsilon \right|\mathcal{F}_{t-1}^N\right] \right]\\
\leq 2 \exp\left( - \frac{N \varepsilon^2}{2|\omega_t|_\infty^2 \left(\sum_{s=\size}^{T}\rho^{\max(t-s,s-r-t,0)}\osc(h_s)\right)^{2}}\right)\eqsp.
\end{multline*}
Lemma \ref{lem:inegEssentielle} therefore yields
\begin{equation*}
\mathbb{P}\left\{ \left| U_{t,T}^N\right| \geq \varepsilon\right\} \leq 2 \exp\left( -\frac{CN \varepsilon^2}{\left(\sum_{s=\size}^{T}\rho^{\max(t-s,s-r-t,0)}\osc(h_s)\right)^{2}}\right)\eqsp.
\end{equation*}
Then $V_{t,T}^N$ is dealt with by using again the Hoeffding inequality and the bounds $0<b_{t,T}^{N,\ell}\leq |\omega_t|_\infty$, where $b_{t,T}^{N,\ell} \eqdef \ewght{t}{N,\ell} \frac{\mathcal{L}_{t,T}\1(\epart{t}{N,\ell})}{|\mathcal{L}_{t,T}\1|_\infty}$:
\begin{multline*}
\mathbb{P}\left[\left|N^{-1}\sum_{\ell=1}^N b_{t,T}^{N,\ell}-\mathbb{E}\left[\left.b_{t,T}^{N,1}\right|\mathcal{F}_{t-1}\right]\right| \geq \varepsilon\right] \\
= \mathbb{E}\left[ \mathbb{P}\left[\left. \left| N^{-1}\sum_{\ell=1}^N \left( b_{t,T}^{N,\ell} - \mathbb{E}\left[ b_{t,T}^{N,\ell} \middle| \mathcal{F}_{t-1}^N\right] \right) \right| \geq \varepsilon \right|\mathcal{F}_{t-1}^N\right] \right]
\leq 2 \exp \left( - CN \varepsilon^2\right)\eqsp.
\end{multline*}
Finally, $W_{t,T}^N$ has been shown in \eqref{Eq:SupIneqForW} to be bounded by a constant depending only on $\sigma_-$, $\sigma_+$, $c_-$, $\underset{t\geq 1}{\sup}|\vartheta_t|_{\infty}$ and $\underset{t\geq 0}{\sup}|\ewght{t}{}|_{\infty} $: $\left|W_{t,T}^N\right| \leq C$ so that
\begin{equation*}
\mathbb{P}\left\{ \left|C_{t,T}^N(\addfunc{T,\size})\right| > \varepsilon \right\} \leq \mathbb{P}\left\{ \left|U_{t,T}^NV_{t,T}^N\right| > \varepsilon/C \right\}\le  \mathbb{P}\left\{ \left|U_{t,T}^N\right| > \varepsilon_{u}\right\} +   \mathbb{P}\left\{ \left|V_{t,T}^N\right| > \varepsilon_{v}\right\}\eqsp,
\end{equation*}
where 
\[
\varepsilon_{u}\eqdef \sqrt{\varepsilon\sum_{s=\size}^{T}\rho^{\max(t-s,s-r-t,0)}\osc(h_s)/C}\quad\mbox{and}\quad \varepsilon_{u}\eqdef \sqrt{\frac{\varepsilon}{C\sum_{s=\size}^{T}\rho^{\max(t-s,s-r-t,0)}\osc(h_s)}}\eqsp.
\]
Therefore,
\[
\mathbb{P}\left\{ \left|C_{t,T}^N(\addfunc{T,\size})\right| > \varepsilon \right\} \leq  4 \exp\left(-\dfrac{CN\varepsilon}{\sum_{s=\size}^{T}\rho^{\max(t-s,s-r-t,0)}\osc(h_s)}\right) \eqsp.
\]
The proof of \eqref{Eq:ExpC} is finally completed by applying Lemma \ref{Lem:SumExpDev} with
\begin{equation*}
X_t = C_{t,T}^N(\addfunc{T,\size}) \eqsp, \quad A=4 \eqsp, \quad B_{t} = \dfrac{CN}{\sum_{s=\size}^{T}\rho^{\max(t-s,s-r-t,0)}\osc(h_s)}\eqsp, \quad \gamma = 1/2\eqsp.
\end{equation*}
\end{proof}

\begin{proof}[Proof of Theorem \ref{Th:ExpIneq} for the FFBS estimator]
The result is obtained by writing
\begin{equation*}
\mathbb{P}\left\{ \left|\error\right| > \varepsilon \right\} \leq \mathbb{P}\left\{ \left|\sum_{t=0}^{T}C_{t,T}^N(\addfunc{T,\size})\right| > \varepsilon/2 \right\} + \mathbb{P}\left\{ \left|\sum_{t=0}^{T}D_{t,T}^N(\addfunc{T,\size})\right| > \varepsilon/2 \right\}\eqsp,
\end{equation*}
and using \eqref{Eq:ExpD} and \eqref{Eq:ExpC}.
\end{proof}

\begin{proof}[Proof of Theorem \ref{Th:ExpIneq} for the FFBSi estimator]
We recall the decomposition used in the proof of Theorem \ref{Th:MainResult} for the FFBSi estimator:
\begin{equation*}
\errorsmall\left[\addfunc{T,\size}\right] = \frac{1}{N} \sum_{\ell=1}^N \sum_{u=0}^T \zeta_u^{N,\ell}\eqsp,
\end{equation*}
where $\errorsmall\left[\addfunc{T,\size}\right]$ is defined by \eqref{eq:errorsmall}.
Since $\left\{\zeta_u^{N,\ell}\right\}_{\ell=1}^{N}$ are $\mathcal{G}_{u,T}^N$ measurable and centered conditionally to $\mathcal{G}_{u+1,T}^N$ using the same steps as in the proof of Proposition \ref{Prop:ExpD}, we get
\begin{equation*}
\mathbb{P}\left\{ \left|\errorsmall\left[\addfunc{T,\size}\right]\right| > \varepsilon \right\} \leq 2 \exp\left( -\dfrac{CN\varepsilon^2}{\Theta_{\size,T}\sum_{s=\size}^{T}\osc(h_{s})^{2}}\right)\eqsp,
\end{equation*}
where $\Theta_{\size,T}$ is defined by \eqref{eq:defTheta}.
The proof is finally completed by writing
\[
\phi_{0:T|T}\left[\addfunc{T,\size}\right]-\widetilde{\phi}_{0:T|T}^N\left[\addfunc{T,\size}\right] = \error + \errorsmall\left[\addfunc{T,\size}\right]\eqsp,
\]
and by using Theorem \ref{Th:ExpIneq} for the FFBS estimator.
\end{proof}

\appendix

\section{Technical results}

\begin{lemma}\label{lem:inegEssentielle}
Assume that $a_N$, $b_N$, and $b$ are random variables defined on the same probability space such that there exist positive constants $\beta$, $B$, $C$, and $M$ satisfying
\begin{enumerate}[(i)]
    \item $|a_N/b_N|\leq M$, $\mathbb{P}$-a.s.\ and  $b \geq \beta$, $\mathbb{P}$-a.s.,
    \item For all $\epsilon>0$ and all $N\geq1$, $\mathbb{P}\left[|b_N-b|>\epsilon \right]\leq B \rme^{-C N \epsilon^2}$,
    \item For all $\epsilon>0$ and all $N\geq1$, $\mathbb{P} \left[ |a_N|>\epsilon \right]\leq B \rme^{-C N \left(\epsilon/M\right)^2}$.
\end{enumerate}
Then,
$$
    \mathbb{P}\left\{ \left| \frac{a_N}{b_N} \right| > \epsilon \right\} \leq B \exp{\left(-C N \left(\frac{\epsilon \beta}{2M} \right)^2 \right)} \eqsp.
$$
\end{lemma}
\begin{proof}
See \TechLem.
\end{proof}

\begin{lemma}\label{Lem:SumExpDev}
For $T\geq 0$, let $\{X_t\}_{t=0}^{T}$ be $(T+1)$ random variables. Assume that there exists a constants $A \geq 1$ and for all $0 \leq t \leq T$, there exists a constant $B_t > 0$ such that and all $\varepsilon > 0$
\begin{equation*}
\mathbb{P}\{ |X_t| > \varepsilon \} \leq A e^{-B_t\varepsilon}\eqsp.
\end{equation*}
Then, for all $0 < \gamma < 1$ and all $\varepsilon > 0$, we have
\begin{equation*}
\mathbb{P}\left\{ \left|\sum_{t=0}^T X_t\right| > \varepsilon \right\} \leq \frac{A}{1-\gamma} e^{-\gamma B\varepsilon/(T+1)}\eqsp,
\end{equation*}
where
\[
B \eqdef \left(\frac{1}{T+1}\sum_{t=0}^T B_t^{-1} \right)^{-1}\eqsp.
\]
\end{lemma}
\begin{proof}
By the Bienayme-Tchebychev inequality, we have
\begin{multline}\label{eq:bienayme}
\mathbb{P}\left\{ \left|\sum_{t=0}^T X_t\right| > \varepsilon \right\} = \mathbb{P}\left\{ \exp\left[\frac{\gamma B}{T+1}\left|\sum_{t=0}^T X_t\right|\right] > e^{\gamma B\varepsilon/(T+1)} \right\} \\
\leq e^{-\gamma B\varepsilon/(T+1)} \mathbb{E}\left[ \exp\left[\frac{\gamma B}{T+1}\left|\sum_{t=0}^T X_t\right|\right] \right] \eqsp.
\end{multline}
It remains to bound the expectation in the RHS of \eqref{eq:bienayme} by $A(1-\gamma)^{-1}$. First, by the Minkowski inequality,
\begin{multline*}
\mathbb{E}\left[ \exp\left[\frac{\gamma B}{T+1}\left|\sum_{t=0}^T X_t\right|\right] \right] = \sum_{q=0}^\infty \frac{\gamma^qB^q}{q!(T+1)^q} \mathbb{E}\left[ \left|\sum_{t=0}^T X_t\right|^q \right]    \\
\leq 1 + \sum_{q=1}^\infty \frac{\gamma^qB^q}{q!(T+1)^{q}}  \left(\sum_{t=0}^T \normL{q}{X_{t}}\right)^{q}\eqsp.
\end{multline*}
Moreover, for $q \geq 1$, $\mathbb{E}\left[ \left|X_t\right|^q \right]$ can be bounded by
\begin{equation*}
\mathbb{E}\left[ \left|X_t\right|^q \right] = \int_0^\infty \mathbb{P}\{ |X_t| > \varepsilon^{1/q} \} \rmd\varepsilon \leq A \int_0^\infty e^{-B_{t}\varepsilon^{1/q}}\rmd\varepsilon = \dfrac{Aq!}{B_{t}^q}   \eqsp,
\end{equation*}
Finally,
\begin{equation*}
\mathbb{E}\left[ \exp\left[\frac{\gamma B}{T+1}\left|\sum_{t=0}^T X_t\right|\right] \right] \leq A \sum_{q=0}^\infty \gamma^q = \frac{A}{(1-\gamma)} \eqsp.
\end{equation*}
\end{proof}

\bibliographystyle{plain}
\bibliography{./FFBSimsbib}

\end{document}